\documentclass[6pt]{article}
\usepackage[latin1]{inputenc}
\usepackage{amsmath,amsthm,amssymb}
\usepackage{graphicx}
\usepackage[all,poly,knot]{xy}
\usepackage[numbers,sort&compress]{natbib}
\usepackage{indentfirst}
\usepackage{latexsym}
\usepackage{graphicx}
\usepackage{flafter}
\usepackage{caption}
\usepackage{textcomp}
\usepackage{hhline}
\usepackage{bigstrut,bigdelim}
\usepackage{rotating}
\usepackage{multicol,multirow}
\usepackage{makeidx}
\usepackage{epsfig}
\usepackage{fancyhdr}
\usepackage{booktabs}
\usepackage{mathrsfs}
\usepackage[all]{xy}
\usepackage{xcolor}
\usepackage{microtype}

\usepackage{color}

\renewcommand{\arraystretch}{1.7}
\renewcommand{\theequation}{\thesection.\arabic{equation}}

\newtheorem{lem}{Lemma}[section]
\newtheorem{thm}{Theorem} [section]
\newtheorem{prop}{Proposition} [section]

\newtheorem{example}{Example} [section]
\newcommand{\Tr}{{\rm Tr}}
\newcommand{\N}{{\rm N}}
\newtheorem*{conflict of interest}{Conflict of interest}

\title{ New permutation polynomials over $\mathbb{F}_{q^2}$
\thanks{Supported By NSF of China No. 12171163, 12571003, 12501006 }}

\author{Xuan Pang$^a$,~ Pingzhi Yuan$^a$\footnote{Corresponding author. E-mail: yuanpz@scnu.edu.cn(P. Yuan).}, ~Danyao Wu$^b$,  ~Huanhuan Guan$^c$ \\
\small  \it $^a$School of  of Mathematical Science, South China Normal University, Guangzhou 510631, China\\
\small \it  $^b$School of Computer Science and Technology, Dongguan University of Technology,
Dongguan 523808, China\\
\small  \it $^c$School of Mathematics and Statistics, Guizhou University of Finance and
Economics, Guiyang 550025, China
}

\date{}

\begin{document}
\baselineskip15pt \maketitle
\renewcommand{\theequation}{\arabic{section}.\arabic{equation}}
\catcode`@=11 \@addtoreset{equation}{section} \catcode`@=12

    \begin{abstract}In this paper, we propose a new method to obtain new permutation polynomials over $\mathbb{F}_{q^2}$. Using this method, we extend many known permutation polynomials, which take the form $\sum_i(x^q-x+\delta)^{s_i}+L(x)$, where $L(x)$ is a $q$-polynomial over $\mathbb{F}_q$ and $\delta\in\mathbb{F}_{q^2}$. We also present an alternative approach for constructing permutation polynomials of the form $x+\gamma\Tr_q^{q^d}(x^{q+1}+x^{2q+2})$ for the cases where $q=2^m$, $2\nmid d$ and $\Tr_q^{q^d}(x)=x+x^q+\dots+x^{q^{d-1}}$.

\end{abstract}


\section{Introduction}

\,\,\, Let $q$ be a prime power, $\mathbb{F}_q$ be the finite field of order $q$, and $\mathbb{F}_q[x]$
be the ring of polynomials in a single indeterminate $x$ over $\mathbb{F}_q$. A polynomial
$f(x) \in\mathbb{F}_q[x]$ is called a {\em permutation polynomial} (PP) of $\mathbb{F}_q$ if it induces
a bijective map from $\mathbb{F}_q$ to itself.

Permutation polynomials over finite fields have attracted significant attention due to their wide-ranging applications, especially in coding theory \cite{Ding,Ding-Zhou,Laigle-Chapuy}, cryptography \cite{Rivest-Shamir-Adelman,Schwenk-Huber}, combinatorial design theory \cite{Ding-Yuan}, and other areas of mathematics and engineering \cite{Lidl-Niederreiter2,Mullen}. Constructing a class of permutation polynomials with a simple form or determining whether a class of polynomials is a permutation polynomial is an interesting and challenging problem.

Helleseth and Zinoviev \cite{Hellzino2003} applied permutation polynomials of the form $(x^2 + x + \delta)^{-2^l}+ x$ to derive new
Kloosterman sums identities over $\mathbb{F}_{2^n}$,  where $\delta\in\mathbb{F}_{2^n}$ and $l = 0, 1$. Motivated by this work, permutation polynomials of the form
\begin{equation}
(x^{p^i}- x + \delta)^s+ L(x)\tag{1}
\end{equation}
over $\mathbb{F}_{p^n}$ were investigated \cite{YuanD07,YuanDW08}, where $n, i, s$ are positive integers and $L(x)$ is a linearized polynomial with coefficients in $\mathbb{F}_p$. Following the approaches in  \cite{YuanD07,YuanDW08}, numerous permutation polynomials with structures similar to (1) have been constructed \cite{Li-hell2013,Liu-chen-zheng2025,
WangWL17,TuZJ15,TuZLH15,XuCX16,XuFZ19,XuLC22,zeng2010,zha2012,zha2016,ZhengC17}. In addition, noting that several permutation polynomials over finite fields of characteristic 2 were constructed in \cite{TuZJ15,zeng2010}, Zeng et al. considered polynomials over $\mathbb{F}_{2^n}$ of the form
 \begin{equation*}
 (x^{2^i}+x+\delta)^{s_1}+(x^{2^i}+x+\delta)^{s_2}+x
 \end{equation*}
 and subsequently presented eight new classes of permutation polynomials with this type \cite{ZengZLL17}. More recently, similar investigations were carried out for polynomials of the form
 \begin{equation*}
 (x^q-x+\delta)^{s_1}+(x^q-x+\delta)^{s_2}+L(x)
 \end{equation*}
 over $\mathbb{F}_{q^2}$ with odd characteristic \cite{Li-Cao2023,LiWLZ18,Liu-Jiang-Zou2025,Liu-xie2021}.


%
%

The main purpose of this paper is to construct several new classes of permutation polynomials over $\mathbb{F}_{q^2}$ with the form
  \begin{equation*}
 (x^{q}-x+\delta)^{s}+\gamma L(x);~~{\rm or}\tag{2}
  \end{equation*}
\begin{equation*}
 (x^{q}-x+\delta)^{s_1}+(x^q-x+\delta)^{s_2}+\gamma L(x),\tag{3}
  \end{equation*}
where $\gamma\in\mathbb{F}_{q^2}^*$, by using a novel approach that differs from  AGW criterion.

The remainder of this paper is structured as follows. In Section 2, we introduce a novel approach for verifying whether a given polynomial is a permutation. This method provides an alternative to classical techniques such as the AGW criterion.
 It also enables the analysis of certain polynomials whose permutation properties  are hard to determine with existing methods.
In Section 3, we apply this method to construct some classes of permutation polynomials of the forms given in (2) and (3) over finite fields. In Section 4,   drawing inspiration from the work of Jiang, Li, and Qu \cite{LiKQ2026}, we present an  approach to constructing permutation polynomials with certain special structures. Finally, Section 5 concludes this paper and outlines directions for future work.

\section{General results}
\,\,\, It is well known that $\mathbb{F}_{q^n}$ and $\mathbb{F}_{q}^n$ are isomorphic as vector spaces over $\mathbb{F}_q$. This isomorphism allows us to study PPs by translating an univariate polynomial $f(x) \in \mathbb{F}_{q^n}[x]$ into  multivariate polynomial map $F = (f_1(x_1,\dots,x_n), \dots, f_n(x_1,\dots,x_n))\in\mathbb{F}_q^n[x_1,\dots,x_n]$ \cite{Evoyan-Kyu1023}. In this section, we give a clear characterization of this correspondence. The arguments presented are elementary and rely only on linear algebra.


\begin{prop}\label{pro1}
Let $\{\alpha_1, \alpha_2, \dots, \alpha_n\}$ be a basis of $\mathbb{F}_{q^n}$ over $\mathbb{F}_q$. For a polynomial $f(x)\in\mathbb{F}_{q^n}[x]$ and $a_i, b_i\in \mathbb{F}_q$, $1\le i\le n$, we let
$$x=(x_1+a_1, \ldots, x_n+a_n)A(\alpha_1, \dots, \alpha_n)^T, $$
$$f(x)=(f_1(x_1, \dots, x_n)+b_1, \dots, f_n(x_1, \dots, x_n)+b_n)B(\alpha_1, \dots, \alpha_n)^T,$$
where $A=(a_{ij})_{n\times n}, B=(b_{ij})_{n\times n}\in M_n(\mathbb{F}_q)$ are invertible matrices over $\mathbb{F}_q$, and $(\alpha_1, \dots, \alpha_n)^T$ denotes the transpose of $(\alpha_1, \dots, \alpha_n)$. Then $f(x)$ is a PP over $\mathbb{F}_{q^n}$ if and only if $(f_1(x_1, \dots, x_n), \dots, f_1(x_1, \dots, x_n))$ permutes $\mathbb{F}_q^n$.
 \end{prop}
\begin{proof} We first prove that if $f(x)$ is a PP over $\mathbb{F}_{q^n}$, then $(f_1(x_1, \dots, x_n), \dots, f_1(x_1, \dots, x_n))$ permutes $\mathbb{F}_q^n$.

 For any element $(c_1, \dots, c_n)\in\mathbb{F}_q^n$. Let $c=(c_1+b_1, \dots, c_n+b_n)B(\alpha_1, \dots, \alpha_n)^T\in\mathbb{F}_{q^n}$. Since $f(x)$ is a PP over $\mathbb{F}_{q^n}$, $f(x)=c$ has a unique solution $x\in\mathbb{F}_{q^n}$, so $x=(y_1, \dots, y_n)(\alpha_1, \dots, \alpha_n)^T$ with $(y_1, \dots, y_n)\in\mathbb{F}_q^n$. Let
$$(x_1, \dots, x_n)=(y_1, \dots, y_n)A^{-1}-(a_1, \dots, a_n).$$
Then
\begin{align*}
f(x)&=f((y_1, \dots, y_n)(\alpha_1, \dots, \alpha_n)^T)\\
 &=f((x_1+a_1, \ldots, x_n+a_n)A(\alpha_1, \dots, \alpha_n)^T)\\
 &=(f_1(x_1, \dots, x_n)+b_1, \dots, f_n(x_1, \dots, x_n)+b_n)B(\alpha_1, \dots, \alpha_n)^T\\
 &=((c_1, \dots, c_n)+(b_1, \dots, b_n))B(\alpha_1, \dots, \alpha_n)^T.
\end{align*}
It follows that $(f_1, \dots, f_n)=(c_1, \dots, c_n)$ has a solution $(x_1, \dots, x_n)\in\mathbb{F}_q^n$ since $\{\alpha_1, \alpha_2, \dots, \alpha_n\}$ is a basis of $\mathbb{F}_{q^n}$ over $\mathbb{F}_q$. Hence $(f_1(x_1, \dots, x_n), \dots, f_1(x_1, \dots, x_n))$ permutes $\mathbb{F}_q^n$.

Conversely, suppose that $(f_1(x_1, \dots, x_n), \dots, f_1(x_1, \dots, x_n))$ permutes $\mathbb{F}_q^n$. For any $c'\in\mathbb{F}_q^n$, we write $c'=(c_1', \dots, c_n')(\alpha_1, \dots, \alpha_n)^T$, and define
$$(c_1, \dots, c_n)=(c_1', \dots, c_n')B^{-1}-(b_1, \dots, b_n).$$
Then $(c_1+b_1, \dots, c_n+b_n)B=(c_1', \dots, c_n')$. Since $(f_1(x_1, \dots, x_n), \dots, f_1(x_1, \dots, x_n))$ permutes $\mathbb{F}_q^n$, $(f_1, \dots, f_n)=(c_1, \dots, c_n)$ has a solution $(x_1, \dots, x_n)\in\mathbb{F}_q^n$, and thus
$$(f_1(x_1, \dots, x_n)+b_1, \dots, f_n(x_1, \dots, x_n)+b_n)B=(c_1', \dots, c_n').$$
By the assumption, we get
$$f((x_1+a_1, \ldots, x_n+a_n)A(\alpha_1, \dots, \alpha_n)^T)=(c_1', \dots, c_n')(\alpha_1, \dots, \alpha_n)^T=c'.$$
Let $(y_1, \dots, y_n)=(x_1+a_1, \ldots, x_n+a_n)A$ and $x=(y_1, \dots, y_n)(\alpha_1, \dots, \alpha_n)^T$. Then we have $f(x)=c'$, i.e., $f(x)=c'$ has a solution $x=(y_1, \dots, y_n)(\alpha_1, \dots, \alpha_n)^T$. Hence $f(x)$ is a PP over $\mathbb{F}_{q^n}$. This completes the proof.\end{proof}

Proposition \ref{pro1} can be rewritten as follows:

\begin{prop}\label{pro-rew}
Let $\{\alpha_1,  \dots, \alpha_n\}$ and $\{\beta_1, \dots, \beta_n\}$  be two bases of $\mathbb{F}_{q^n}$ over $\mathbb{F}_q$. For a polynomial $f(x)\in\mathbb{F}_{q^n}[x]$ and $a_i\in \mathbb{F}_q$, $1\le i\le n$, we let
$$x=(x_1+a_1, \ldots, x_n+a_n)(\alpha_1, \dots, \alpha_n)^T, $$
$$f(x)=(f_1(x_1, \dots, x_n), \dots, f_n(x_1, \dots, x_n))(\beta_1, \dots, \beta_n)^T+c, $$
where $c\in\mathbb{F}_{q^n}$. Then $f(x)$ is a PP over $\mathbb{F}_{q^n}$ if and only if $(f_1(x_1, \dots, x_n), \dots, f_1(x_1, \dots, x_n))$ permutes $\mathbb{F}_q^n$. \end{prop}

Next, we present a result for polynomials of the form $x+g(\Tr_q^{q^n}(x))$  over $\mathbb{F}_{q^n}$. This result can be viewed as a special case of \cite[Corollary 3.4]{YuanDing2011} and  can undoubtedly be derived from the  AGW criterion. However, we provide an alternative proof below  based on the Proposition \ref{pro-rew} above.  Note that the {\em trace function}
and {\em norm function} from $\mathbb{F}_{q^n}$ to $\mathbb{F}_{q}$ are defined respectively by
$$\Tr_{q}^{q^n}(x)= x+x^q+\cdots+x^{q^{n-1}},\quad\N_{q}^{q^n}(x)=x^{(q^n-1)/(q-1)}.$$
In particular, when $q=p$ is a prime, we call $\Tr_{q}^{q^n}(\cdot)$ the {\em absolute trace function}.

\begin{example}
Let $q$ be a prime power and $g(x)=\sum_{i=1}^{q-1}a_ix^i\in\mathbb{F}_{q^n}[x]$. Then the mapping $f(x)=x+g(\Tr_{q}^{q^n}(x))$ permutes $\mathbb{F}_{q^n}$ if and only if $h(x)=x+\sum_{i=1}^{q-1}\Tr_{q}^{q^n}(a_i)x^i$ permutes $\mathbb{F}_q$.
\end{example}
\begin{proof}
Let $\{\alpha_1,\dots,\alpha_n\}$ and $\{1,\beta_2,\dots,\beta_n\}$ be a dual pair of ordered bases of $\mathbb{F}_{q^n}$ over $\mathbb{F}_q$. Then for any $x\in\mathbb{F}_{q^n}$, we have
$$x=(x_1,\dots,x_n)(\alpha_1,\dots,\alpha_n)^T,~~{\rm with}~~ x_i\in\mathbb{F}_q, ~{\rm for} ~1\leq i\leq n.$$  Now assume that $a_j=(a_{j1},\dots,a_{jn})(\alpha_1,\dots,\alpha_n)^T$ with $a_{ji}\in\mathbb{F}_q$ for $1\leq i\leq n$ and $1\leq j\leq q-1$.
Hence
\begin{align*}
f(x)&=x+a_1\Tr_{q}^{q^n}(x)+\dots+a_{q-1}\left(\Tr_{q}^{q^n}(x)\right)^{q-1}\\
    &=\sum_{i=1}^n\alpha_ix_i+\sum_{i=1}^{n}\alpha_ia_{1i}\Tr_{q}^{q^n}(x)+\cdots+\sum_{i=1}^n\alpha_ia_{(q-1)i}\left(\Tr_{q}^{q^n}(x)\right)^{q-1}\\
    &=\left(x_1+\sum_{j=1}^{q-1}a_{j1}\left(\Tr_{q}^{q^n}(x)\right)^j\right)\alpha_1+\left(x_2+\sum_{j=1}^{q-1}a_{j2}\left(\Tr_{q}^{q^n}(x)\right)^j\right)\alpha_2+\cdots+\left(x_n+\sum_{j=1}^{q-1}a_{jn}\left(\Tr_{q}^{q^n}(x)\right)^j
   \right)\alpha_n.\\
\end{align*}
It follows from Proposition \ref{pro-rew} that $f(x)$ is a PP over $\mathbb{F}_{q^n}$ if and only if
\begin{equation*}
\left(x_1+\sum_{j=1}^{q-1}a_{j1}\left(\Tr_{q}^{q^n}(x)\right)^j,x_2+\sum_{j=1}^{q-1}a_{j2}\left(\Tr_{q}^{q^n}(x)\right)^j,\cdots,x_n+\sum_{j=1}^{q-1}a_{jn}\left(\Tr_{q}^{q^n}(x)\right)^j\right)\tag{4}
\end{equation*}
permutes $\mathbb{F}_q^n$. That is
$$\left(x_1+\sum_{j=1}^{q-1}\Tr_{q}^{q^n}(a_j)x_1^j,x_2+\sum_{j=1}^{q-1}a_{j2}x_1^j,\cdots,x_n+\sum_{j=1}^{q-1}a_{jn}x_1^j\right),$$
since a simple computation derives $\Tr_{q}^{q^n}(x)=x_1$ and $\Tr_{q}^{q^n}(a_j)=a_{j1}$,  where $1\leq j\leq q-1$.
Observe that $h(x_1)=x_1+\sum_{j=1}^{q-1}\Tr_{q}^{q^n}(a_j)x_1^j$ is a univariate polynomial in $x_1$, and each of the remaining component functions has the form $x_i+h_i(x_1)$. Consequently, the system of $(4)$ permutes $\mathbb{F}_{q}^n$ if and only if $h(x_1)$ permutes $\mathbb{F}_q$. It follows that $f(x)$ is a PP of $\mathbb{F}_{q^n}$ if and only if $h(x)$ permutes $\mathbb{F}_q$. This completes the proof.
\end{proof}

\section{PPs of the form $\sum_i(x^q-x+\delta)^{s_i}+\gamma L(x)$ over $\mathbb{F}_{q^2}$}
In this section, we construct some new classes of PPs with the form $(2)$ or $(3)$ over $\mathbb{F}_{q^2}$. As shown in Proposition \ref{pro-rew}, when $n=2$, we have

\begin{lem}\label{lem-n=2}
Let $\{\alpha_1, \alpha_2\}$ and $\{\beta_1, \beta_2\}$ be two bases of $\mathbb{F}_{q^2}$ over $\mathbb{F}_q$. For a polynomial $f(x)\in\mathbb{F}_{q^2}[x]$ and elements $a, b, c, d\in\mathbb{F}_q$, let $x=(y+a)\alpha_1+(z+b)\alpha_2$, where $y,z$ are variables over $\mathbb{F}_q$. Then we have
$$f(x)=f((y+a)\alpha_1+(z+b)\alpha_2)=g_1(y, z)\beta_1+g_2(y, z)\beta_2+ \alpha$$
 where $g_i(y, z)\in\mathbb{F}_q[x, y]$ for $i=1, 2$,
 and $\alpha\in\mathbb{F}_{q^2}$ is a constant which does not depend on $y$ and $z$. Moreover, $f(x)$ is a PP over $\mathbb{F}_{q^2}$ if and only if $(g_1(y, z),g_2(y, z))$ permutes $\mathbb{F}_q^2$.\end{lem}

\subsection{Odd characteristic}

Wu and Yuan in \cite{WuYuan2024} investigated the permutation properties of the polynomials $f(x)=(x^{3^m}-x+\delta)^{2\cdot3^m+1}+x$ over $\mathbb{F}_{3^{2m}}$, and derived the necessary and sufficient conditions for such polynomials to be permutations. We now present a more general result as follows.

\begin{thm}\label{Wuyuan}
For an odd prime power $q$, let $\gamma,\delta\in\mathbb{F}_{q^2}$ with $\gamma\neq0$. Then the polynomial
$$f(x)=(x^q-x+\delta)^{q+2}+\gamma x$$
is a PP of $\mathbb{F}_{q^2}$ if and only if one of the following holds:
\begin{enumerate}
\item[{\rm(i)}] $\gamma\in\mathbb{F}_q^*, 3\mid q$ and $\left(\Tr_{q}^{q^2}(\delta)\right)^2-\Tr_{q}^{q^2}(\gamma)$ is a square;
\item[{\rm(ii)}] $\gamma\in\mathbb{F}_q^*, q\equiv2\pmod3$ and $\left(\Tr_{q}^{q^2}(\delta)\right)^2=Tr_{q}^{q^2}(\gamma)$;

\item[{\rm(iii)}] $\gamma\in\mathbb{F}_{q^2}^*\backslash\mathbb{F}_{q}$, $\Tr_q^{q^2}(\delta)=\Tr_{q}^{q^2}(\gamma)=0$;

\item[{\rm(iv)}] $\gamma\in\mathbb{F}_{q^2}^*\backslash\mathbb{F}_{q}$, $\Tr_q^{q^2}(\delta)=0$,
    $\Tr_{q}^{q^2}(\gamma)\neq0$, $3\mid q$ and   $-\frac{\N_{q}^{q^2}(\gamma)}{\Tr_{q}^{q^2}(\gamma)}$ is a square;

\item[{\rm(v)}] $\gamma\in\mathbb{F}_{q^2}^*\backslash\mathbb{F}_{q}$,  $\Tr_q^{q^2}(\delta)\neq0$, $\Tr_{q}^{q^2}(\gamma)\neq0$, $q\equiv2\pmod3$ and
     $$\left(\Tr_{q}^{q^2}(\delta)\Tr_{q}^{q^2}(\gamma)\right)^2=\N_{q}^{q^2}(\gamma)\left(\left(\Tr_{q}^{q^2}(\delta)\right)^2+3\Tr_{q}^{q^2}(\gamma)\right).$$
 \end{enumerate}
\end{thm}
\begin{proof}
  Since $q$ is an odd prime power,  there is a non-square element $u\in\mathbb{F}_{q}$. Let $\alpha$ be an element of $\mathbb{F}_{q^2}$ with $\alpha^2=u$. Then $\{1, \alpha\}$ forms a basis of $\mathbb{F}_{q^2}$ over $\mathbb{F}_q$. Let $\delta=a+b\alpha$, $\gamma=c+d\alpha$ and $x=y-(z-b)\alpha/2$, where $a,b,c,d,y,z\in\mathbb{F}_q$. Then $x^q-x+\delta=a+z\alpha$. We have
  \begin{align*}
f(x)=f(y-(z-b)\alpha/2)&=\left(a+z\alpha\right)^{q+2}+(c+d\alpha)(y-(z-b)\alpha/2)\\
  &=\left(a-z\alpha\right)\left(a+z\alpha\right)^{2}+(c+d\alpha)(y-(z-b)\alpha/2)\\
  &=cy-uaz^2-\frac{ud}{2}z+a^3+\frac{ubd}{2}+\left(dy-uz^3+\left(a^2-\frac{c}{2}\right)z+\frac{bc}{2}\right)\alpha.
\end{align*}
Hence
\begin{equation*}
\begin{cases}
g_1(y,z)=cy-uaz^2-\frac{ud}{2}z,\\
g_2(y,z)=dy-uz^3+\left(a^2-\frac{c}{2}\right)z.
\end{cases}
\end{equation*}
By Lemma \ref{lem-n=2}, it suffices to analyze the permutation behavior of $(g_1(y,z),g_2(y,z))$. We therefore distinguish three cases, as $\gamma \neq 0$.

Case 1: Assume $c\neq0$ and $d=0$. Observe that $g_2(y,z)$ is a polynomials of $z$ and the $y$-part of $g_1(y,z)$ is $cy$. Therefore $(g_1(y,z),g_2(y,z))$ permutes $\mathbb{F}_q^2$ if and only if $g_2(y,z)=g_2(z)$ permutes $\mathbb{F}_q$. In this case, the normalized form of $g_2(z)$ is
$$z^3-\frac{2a^2-c}{2u}z.$$
By consulting Table 7.1 in \cite{LN97}, the polynomial $g_2(z)$ is a PP over $\mathbb{F}_q$ if and only if $3\mid q$ and $\frac{2a^2-c}{2u}$ is a non-square, that is, $\frac{2a^2-c}{2}$ is a square, or $q\equiv2\pmod3$ and  $2a^2-c=0$.

Case 2: Assume $c=0$ and $d\neq0$. Note that $g_1(y,z)=g_1(z)$ becomes a polynomial in $z$ alone, and the $y$-part of $g_2(y,z)$ is $dy$. Clearly,  $(g_1(y,z),g_2(y,z))$ permutes $\mathbb{F}_q^2$ if and only if $g_1(z)=-uaz^2-\frac{ud}{2}z$ permutes $\mathbb{F}_q$, which is equivalent to $a=0$.

Case 3: Assume $cd\neq0$. Observe that the $y$-part of $g_1(y,z)$ and $g_2(y,z)$ are $cy$ and $dy$, respectively. Then $(g_1(y,z),g_2(y,z))$ permutes $\mathbb{F}_q^2$ if and only if
 \begin{align*}
  (cu)^{-1}h(z)&=(cu)^{-1}\left(d\cdot g_1(y,z)-c\cdot g_2(y,z)\right)\\
    &=(cu)^{-1}\left(cuz^3-aduz^2-\left(\frac{d^2u}{2}+a^2c-\frac{c^2}{2u}\right)z\right)\\
    &=z^3-\frac{ad}{c}z^2-(\frac{d^2}{2c}+\frac{a^2}{u}-\frac{c}{2u})z
 \end{align*}
 permutes $\mathbb{F}_q$ since $cu^{-1}\neq0$. By Table 7.1 in \cite{LN97}  $(cu)^{-1}h(z)$ permutes $\mathbb{F}_{q}$ only in the following two cases. If $3\mid q$, then  $(cu^{-1})h(z)$ permutes $\mathbb{F}_{q}$  if and only if $a=0$ and $\frac{d^2u-c^2}{2c}$ is a square in $\mathbb{F}_q$. If $3\nmid q$, the normalized form of $ (cu^{-1})h(z)$ becomes
 $$z^3-\left(\frac{a^2d^2}{3c^2}+\frac{d^2}{2c}+\frac{a^2}{u}-\frac{c}{2u}\right)z.$$
  In this case, $(cu)^{-1}h(z)$ permutes $\mathbb{F}_{q}$ if and only  if $2ua^2d^2=3c(c^2-d^2u-2a^2c)$ and $q\equiv2\pmod3$, where the equation is a simplified form of $$\frac{a^2d^2}{3c^2}+\frac{d^2}{2c}+\frac{a^2}{u}-\frac{c}{2u}=0.$$

Finally, a direct computation yields $2a=\Tr_{q}^{q^2}(\delta)$, $2c=\Tr_{q}^{q^2}(\gamma)$ and $d^2=\frac{c^2-\N_{q}^{q^2}(\gamma)}{u}$. With these relations, the proof is completed.
\end{proof}

In \cite[Proposition 1]{LiWLZ18}, Li et al. completely characterized PPs of the form $(x^q-x+\delta)^{2q}+x$ over $\mathbb{F}_{q^2}$. Following this work,  Xu, Luo and Cao presented an analogous result for $(x^q-x+\delta)^{2}+x$ over $\mathbb{F}_{q^2}$. We now extend these results by providing a clear characterization for the more general form $(x^q-x+\delta)^{s}+\gamma x$ over $\mathbb{F}_{q^2}$, where $s\in\{2,2q\}$. The proofs follow a similar line to those in Theorem \ref{Wuyuan}, and we omit the details here.

\begin{thm}\label{Li_2}
For an odd prime power $q$, let $\gamma,\delta\in\mathbb{F}_{q^2}$ with $\gamma\neq0$. Then the polynomial
$$f(x)=(x^q-x+\delta)^{2}+\gamma x$$
is a PP of $\mathbb{F}_{q^2}$ if and only if $\gamma\in\mathbb{F}_q^*$ and $\Tr_{q}^{q^2}(\delta)-\Tr_{q}^{q^2}(\gamma)/4\neq0$.
 \end{thm}

\begin{thm}\label{Li_2q}
For an odd prime power $q$, let $\gamma,\delta\in\mathbb{F}_{q^2}$ with $\gamma\neq0$. Then the polynomial
$$f(x)=(x^q-x+\delta)^{2q}+\gamma x$$
is a PP of $\mathbb{F}_{q^2}$ if and only if $\gamma\in\mathbb{F}_q^*$ and $\Tr_{q}^{q^2}(\delta)+\Tr_{q}^{q^2}(\gamma)/4\neq0$.
 \end{thm}

\begin{thm}\label{q+2_2q+1}
For an odd prime power $q$, let $\gamma,\delta\in\mathbb{F}_{q^2}$ with $\gamma\neq0$. Then the polynomial
$$f(x)=(x^q-x+\delta)^{q+2}+(x^q-x+\delta)^{2q+1}+\gamma x$$
is a PP of $\mathbb{F}_{q^2}$ if and only if $\gamma\in\mathbb{F}_q^*$; or $\gamma\in\mathbb{F}_{q^2}^*\backslash\mathbb{F}_q$ and $\Tr_{q}^{q^2}(\delta)=0$.
\end{thm}
\begin{proof}
Choose a quadratic non-residue $u\in\mathbb{F}_q$ and let $\alpha^2=u$. Then $\{1,\alpha\}$ is a basis of $\mathbb{F}_{q^2}$ over $\mathbb{F}_q$. Let $\delta=a+b\alpha$, $\gamma=c+d\alpha$ and $x=y-(z-b)\alpha/2$, where $a,b,c,d,x,y\in\mathbb{F}_q$. A simple computation shows that $2a=\Tr_q^{q^2}(\delta)$ and $x^q-x+\delta=a+z\alpha$. Therefore we have
\begin{align*}
f(x)=f(y-(z-b)\alpha/2)&=(a+z\alpha)^{q+2}+(a+z\alpha)^{2q+1}+(c+d\alpha)(y-(z-b)\alpha/2)\\
  &=(a-z\alpha)(a+z\alpha)^2+(a-z\alpha)^2(a+z\alpha)+(c+d\alpha)(y-(z-b)\alpha/2)\\
  &=2a(a^2-z^2u)+(c+d\alpha)(y-(z-b)\alpha/2)\\
  &=cy-2auz^2-\frac{ud}{2}z+2a^3+\frac{ubd}{2}+\left(dy-\frac{c}{2}z+\frac{bc}{2}\right)\alpha.
\end{align*}
Hence
\begin{equation*}
\begin{cases}
g_1(y,z)=cy-2auz^2-\frac{ud}{2}z,\\
g_2(y,z)=dy-\frac{c}{2}z.
\end{cases}
\end{equation*}
The task now reduces to studying the permutation properties of $(g_1(y,z),g_2(y,z))$. Note that $c$ and $d$ cannot both be zero since $\gamma\neq0$.
Suppose first that $c\neq0$ and $d=0$. Then $g_2(y,z)=-\frac{c}{2}z$, which is clearly a PP over $\mathbb{F}_q$. Thus in this case, $(g_1(y,z),g_2(y,z))$ must be a permutation over $\mathbb{F}_{q}^2$. Similarly, assume that $c=0$ and $d\neq0$, we have $g_1(y,z)=-2auz^2-\frac{ud}{2}z$, which permutes $\mathbb{F}_q$ if and only if $a=0$. So $(g_1(y,z),g_2(y,z))$ is a permutation of $\mathbb{F}_{q}^2$ if and only if $a=0$. Finally suppose that $cd\neq0$. Consider $h(z)=dg_1(y,z)-cg_2(y,z)=-2aduz^2+(\frac{c^2}{2}-\frac{d^2u}{2})z$. It is straightforward to see that $h(z)$ permutes $\mathbb{F}_q$ if and only if $a=0$. It follows that $(g_1(y,z),g_2(y,z))$ permutes $\mathbb{F}_{q}^2$ if and only if $a=0$. In conclusion, we complete the proof.
\end{proof}

\begin{thm}\label{degree5}
Let $q$ be an odd prime power, and let $\gamma,\delta\in\mathbb{F}_{q^2}$ with $\gamma\neq0$. Then the polynomial
$$f(x)=(x^q-x+\delta)^{q+4}+(x^q-x+\delta)^{5}+\gamma x$$
is a PP of $\mathbb{F}_{q^2}$ if and only if  one of the following holds:
\begin{enumerate}
\item[{\rm(i)}] $\Tr_q^{q^2}(\delta)=0$;
\item[{\rm(ii)}] $\Tr_q^{q^2}(\delta)\neq0$, $\gamma\in\mathbb{F}_{q}^*$,  $3\mid q$ and $\Tr_{q}^{q^2}(\gamma)/2-\left(\Tr_{q}^{q^2}(\delta)\right)^4$ is a square;
\item[{\rm(iii)}] $\Tr_q^{q^2}(\delta)\neq0$, $\gamma\in\mathbb{F}_{q}^*$,  $q\equiv2\pmod3$ and $2\left(\Tr_{q}^{q^2}(\delta)\right)^4=\Tr_{q}^{q^2}(\gamma)$.
\end{enumerate}
\end{thm}
 \begin{proof}
Choose a quadratic non-residue $u\in\mathbb{F}_q$ and let $\alpha^2=u$. Then $\{1,\alpha\}$ is a basis of $\mathbb{F}_{q^2}$ over $\mathbb{F}_q$. Let $\delta=a+b\alpha$, $\gamma=c+d\alpha$ and $x=y-(z-b)\alpha/2$, where $a, b, c, d,x,y \in\mathbb{F}_q$. We have
 $x^q-x+\delta=a+z\alpha$ and
\begin{align*}
f(y-(z-b)\alpha/2)&=(a+z\alpha)^{q+4}+(a+z\alpha)^{5}+(c+d\alpha)(y-(z-b)\alpha/2)\\
  &=2a(a+z\alpha)^4+(c+d\alpha)(y-(z-b)\alpha/2)\\
  &=cy+2au^2z^4+12a^3uz^2-\frac{ud}{2}z+2a^5+\frac{ubd}{2}+\left(dy+8a^2uz^3+\left(8a^4-\frac{c}{2}\right)z+\frac{bc}{2}\right)\alpha.
\end{align*}
Hence
 \begin{equation*}
\begin{cases}
g_1(y,z)=cy+2au^2z^4+12a^3uz^2-\frac{ud}{2}z,\\
g_2(y,z)=dy+8a^2uz^3+\left(8a^4-\frac{c}{2}\right)z.
\end{cases}
\end{equation*}
According to Lemma \ref{lem-n=2}, it suffices to study the permutation behavior of $(g_1(y,z),g_2(y,z))$. If $a=0$ (i.e., $\Tr_q^{q^2}(\delta)=0$), it is easy to verify that $(g_1(y,z),g_2(y,z))=(cy-\frac{ud}{2}z,dy-\frac{c}{2}z)$ is a permutation over $\mathbb{F}_q^2$. In the following, we focus on the case where $a\neq0$. Since $\gamma\neq0$, it follows that  $c$ and $d$ cannot both be zero. We therefore distinguish three cases in the following analysis.

Assume $c\neq0$ and $d=0$. Observe that $g_2(y,z)=g_2(z)$ is a polynomial of $z$ and the $y$-part of $g_1(y,z)$ is $cy$. Therefore $(g_1(y,z),g_2(y,z))$ permutes $\mathbb{F}_q^2$ if and only if $g_2(z)$ is a PP of $\mathbb{F}_q$. Since $u$ is a quadratic non-residue, the normalized form of $g_2(z)$ is
$$z^3-\left(\frac{c}{16a^2u}-\frac{a^2}{u}\right)z.$$
By Table 7.1 in \cite{LN97}, it follows that $g_2(z)$ is a PP over $\mathbb{F}_q$ if and only if $3\mid q$ and $\frac{c}{16a^2u}-\frac{a^2}{u}$ is a non-square in $\mathbb{F}_q$, that is, $c-16a^4$ is a square in $\mathbb{F}_q$, or $q\equiv2\pmod3$ and $c-16a^4=0$.

Assume $c=0$ and $d\neq0$. Observe that $g_1(y,z)=g_1(z)$ is a polynomial of $z$ and the $y$-part of $g_2(y,z)$ is $dy$. Thus $(g_1(y,z),g_2(y,z))$ permutes $\mathbb{F}_q^2$ if and only if $g_1(z)$ is a PP of $\mathbb{F}_q$. However, by Table 7.1 in \cite{LN97}, $g_1(z)$ cannot be a PP over $\mathbb{F}_q$, since it is impossible to eliminate quadratic term and leave only the quartic and linear terms in its expression.


 Assume $cd\neq0$. Observe that the $y$-part of $g_1(y,z)$ and $g_2(y,z)$ are $cy$ and $dy$, respectively. Thus we have $(g_1(y,z),g_2(y,z))$ permutes $\mathbb{F}_q^2$ if and only if
\begin{align*}
(2adu^2)^{-1}h(z)&=(2adu^2)^{-1}(d\cdot g_1(y,z)-c\cdot g_2(y,z))\\
   &=(2adu^2)^{-1}\left(2adu^2z^4-8a^2cuz^3+12a^3duz^2+\frac{c^2-d^2u-16a^4c}{2}z\right)\\
   &=z^4-\frac{4ac}{du}z^3+\frac{6a^2}{u}z^2+\frac{c^2-d^2u-16a^4c}{4adu^2}z
\end{align*}
permutes $\mathbb{F}_q$. By Table 7.1 in \cite{LN97}, it suffices to consider the  normalized form of $(2adu^2)^{-1}h(z)$ as follows
$$z^4+\frac{6a^2(d^2u-c^2)}{d^2u^2}z^2+\left(\frac{c^2-d^2u-16a^4c}{4adu^2}-\frac{a^3c^3}{d^3u^3}+\frac{12a^3c}{du^2}\right)z.$$
Since $q$ is odd, $a\neq0$ and $d^2u-c^2\neq0$,  we have that $h(z)$ cannot be a PP over $\mathbb{F}_{q}$.

Finally, substituting the relations  $2a=\Tr_{q}^{q^2}(\delta)$, $2c=\Tr_{q}^{q^2}(\gamma)$ and $d^2=\frac{c^2-\N_{q}^{q^2}(\gamma)}{u}$ into the preceding deductions, we conclude the proof.
 \end{proof}

{\bf Remark:} Among the five classes of PPs presented in Theorems \ref{lem-n=2}-\ref{degree5}, we specifically investigate those with the parameter $\gamma$ belonging to $\mathbb{F}_{q^2}^*$.  To the best of our knowledge, no previous work has analyzed these cases relying solely on the AGW criterion. This is primarily because the commutative diagram required by the AGW criterion breaks down when $\gamma \in \mathbb{F}_{q^2}$. 

\medskip

In 2023, Li and Cao provided a complete characterization for the PPs having the form $(x^{2^m}-x+\delta)^{s_1}+(x^{2^m}-x+\delta)^{s_2}+x$ over $\mathbb{F}_{2^{2m}}$, where $s_1\in\{2^{m+1}+1,2^m+2\}$ and $s_2\in\{3\cdot2^m+2,2^{m+1}+3\}$ (see \cite[Proposition 5]{Li-Cao2023}).
We now consider such polynomials in finite fields of odd characteristic. For brevity, only one of these cases is presented below, as the results for the others are analogous. From now on, we primarily focus on  $\gamma\in\mathbb{F}_q^*$. The cases $\gamma\in\mathbb{F}_{q^2}^*\backslash\mathbb{F}_q$ can be handled using similar arguments as previously.

\begin{thm}\label{Licao2023th}
Let $q$ be an odd prime power, $\delta\in\mathbb{F}_{q^2}$ and $\gamma\in\mathbb{F}_q^*$. Then given that $2a=\Tr_{q}^{q^2}(\delta)$, the polynomial
$$f(x)=(x^q-x+\delta)^{2q+1}+(x^q-x+\delta)^{3q+2}+\gamma x$$
is a PP of $\mathbb{F}_{q^2}$ if and only if one of the following holds:
\begin{enumerate}
\item[\rm{(i)}] $q=9$, $a=\pm1$ and $\gamma=1$;
\item[\rm{(ii)}]$q=13$ and $(a,\gamma)\in\{(0,6),(\pm1,11),(\pm2,4),(\pm5,6)\}$;
\item[{\rm(iii)}] $q\not\equiv1\pmod5$ and $a^2=-\gamma=-1/2$;
  \item[\rm{(iv)}]  $q\equiv\pm2\pmod5$ and $2a^4+2a^2+5\gamma=2$;
 \item[\rm{(v)}] $5\mid q$ and  one of the following occurs:
       (a) $a^2=-1/2$ and $(1-2\gamma)/4$ is a fourth power or not a square in $\mathbb{F}_q$;
    (b) $(2a^2+1)/2$ is a square in $\mathbb{F}_q$ and $2\gamma=1$.
\end{enumerate}
\end{thm}
\begin{proof}
Choose a quadratic non-residue $u\in\mathbb{F}_q$ and let $\alpha^2=u$. Then $\{1,\alpha\}$ is a basis of $\mathbb{F}_{q^2}$ over $\mathbb{F}_q$. Let $\delta=a+b\alpha$ and $x=y-(z-b)\alpha/2$, where $a,b,y,z\in\mathbb{F}_q$. A straightforward calculation gives that $2a=\Tr_{q}^{q^2}(\delta)$ and $x^q-x+\delta=a+z\alpha$. Therefore we have
\begin{align*}
f(x)=f(y-(z-b)\alpha/2)&=(a+z\alpha)^{2q+1}+(a+z\alpha)^{3q+2}+\gamma(y-(z-b)\alpha/2)\\
  &=(a-z\alpha)^2(a+z\alpha)+(a-z\alpha)^3(a+z\alpha)^2+\gamma(y-(z-b)\alpha/2)\\
  &=\gamma y+au^2z^4-au(2a^2+1)z^2+a^5+a^3\\
     &~~~~+\left(-u^2z^5+(2a^2+1)uz^3-(a^4+a^2+\gamma/2)z+b\gamma/2\right)\alpha.
\end{align*}
Hence
\begin{equation*}
\begin{cases}
g_1(y,z)=\gamma y+au^2z^4-au(2a^2+1)z^2,\\
g_2(y,z)=-u^2z^5+(2a^2+1)uz^3-(a^4+a^2+\gamma/2)z.
\end{cases}
\end{equation*}
From Lemma \ref{lem-n=2}, we focus on the permutation behavior of $(g_1(y,z),g_2(y,z))$. Observe that $g_2(y,z)=g_2(z)$ is a polynomial of $z$ and the $y$-part of $g_1(y,z)$ is $\gamma y$. Therefore, $(g_1(y,z),g_2(y,z))$ permutes $\mathbb{F}_q^2$ if and only if $g_2(z)$ permutes $\mathbb{F}_q$. Since $u$ is a quadratic non-residue, the expression for $g_2(z)$ can be rewritten in its normalized form:
$$z^5-\frac{2a^2+1}{u}z^3+\frac{2a^4+2a^2+\gamma}{2u^2}z.$$
By Table $7.1$ in \cite{LN97}, we assert that $g_2(z)$ is a PP of $\mathbb{F}_q$ if and only if one of the following occurs:

\begin{enumerate}
\item[(1)] $q\not\equiv1\pmod5$ and $2a^2+1=2a^4+2a^2+\gamma=0$, i.e., $a^2=-\gamma=-1/2$.
 \item[(2)]  $5\mid q$,
$2a^2+1=0$ and
$$\frac{1-2\gamma}{4u^2}$$ is not a fourth power in $\mathbb{F}_{q}$. That is $(1-2\gamma)/4$ is a fourth power or not a square in $\mathbb{F}_q$.
\item[(3)]$q=9, 2a^2+1=0$ and $$\left(\frac{2a^4+2a^2+\gamma}{4u^2}\right)^2=2.$$
That is $a=\pm 1$ and $(1+\gamma)^2=2u^4$ as Char($\mathbb{F}_q$)=3. We now determine the value of $\gamma$.
 We first claim that $u^4=2$.
Let $g$ be a generator of the multiplicative group $\mathbb{F}_9^*$. Then an element is a quadratic non-residue in $\mathbb{F}_9$ if and only if it is of the form $g^k$ with odd $k$, equivalently, $\gcd(k,8)=1$. It follows that any quadratic non-residue in $\mathbb{F}_9$ has order $8/\gcd(k,8)=8$, implying $u^4=-1$. Therefore $u^4=2$ in $\mathbb{F}_9$. Consequently, we conclude that for $q=9$, $a=\pm1$ and  $\gamma=1$ (as $\gamma\neq0$), the polynomial $g_2(z)$ is a PP of $\mathbb{F}_{q}$.

\item[(4)] $q\equiv\pm2\pmod5$ and
$$\frac{(2a^2+1)^2}{u^2}=5\cdot\frac{2a^4+2a^2+\gamma}{2u^2},$$
that is, $2a^4+2a^2+5\gamma=2$.
\item[(5)]$q=13,-(2a^2+1)$ is a square in $\mathbb{F}_q$ and
$$3\cdot\frac{(2a^2+1)^2}{u^2}=\frac{2a^4+2a^2+\gamma}{2u^2},$$
that is, $9a^4+9a^2+6=\gamma$. We next determine the values of $(a,\gamma)$. The set of square in $\mathbb{F}_{13}$ is $\{0,1,3,4,9,10,12\}$. Since $a^2$ is a square in $\mathbb{F}_{13}$, $2a^2+1$ take values in $\{1,3,7,9,6,8,12\}$. As $-(2a^2+1)\equiv12(2a^2+1)\pmod{13}$ must also be a square and $12$ is a square in $\mathbb{F}_{13}$, it follows that that $2a^2+1\in\{1,3,9,12\}$.
We now determine the corresponding values of $a$ and $\gamma$:
If $2a^2+1=1$, then $a=0$ and $\gamma=6$. If $2a^2+1=3$, then $a=\pm1$ and $\gamma=11$. If $2a^2+1=9$, then $a=\pm2$ and $\gamma=4$. If $2a^2+1=12$, then $a=\pm5 $ and $\gamma=6$. We conclude that for $q=13$, $(a,\gamma)\in\{(0,6),(\pm1,11),(\pm2,4),(\pm5,6)\}$, the polynomial $g_2(z)$ is a PP over $\mathbb{F}_{q}$. Thus $f(x)$ is a PP over $\mathbb{F}_{q^2}$.
\item[(6)]$ 5\mid q$, $(2a^2+1)/2$ is a square in $\mathbb{F}_q$ and
$$\left(\frac{2a^2+1}{2u}\right)^2=\frac{2a^4+2a^2+\gamma}{2u^2},$$
that is, $\gamma=1/2$.

\end{enumerate}
%
%
\end{proof}

With the same method as in Theorem \ref{Licao2023th}, we propose more classes of PPs with the form $(x^q-x+\delta)^{s_1}+(x^q-x+\delta)^{s_2}+\gamma x$, where $\gamma\in\mathbb{F}_q^\ast$. The detailed proofs are omitted here and provided in the appendix.

\begin{thm}\label{Appx.1}
Let $q$ be an odd prime power, $\delta\in\mathbb{F}_{q^2}$ and $\gamma\in\mathbb{F}_q^*$. Then the polynomial
$$f(x) = (x^q-x+\delta)^{2q+3}+(x^q-x+\delta)^{2q}+\gamma x$$
is a PP of $\mathbb{F}_{q^2}$ if and only if  one of the following holds:
\begin{enumerate}
\item[{\rm(i)}] $5\mid q$, $\Tr_{q}^{q^2}(\delta)=0$ and $\gamma/2$ is a fourth power or not a square in $\mathbb{F}_q$;
\item[{\rm(ii)}] $q=9$, $\Tr_{q}^{q^2}(\delta)=0$ and $\gamma=\pm1$;
\item[{\rm(iii)}] $q\equiv\pm2\pmod{5}$, $\Tr_{q}^{q^2}(\delta)\ne0$, and
    $\left(\Tr_{q}^{q^2}(\delta)\right)^4-80\Tr_{q}^{q^2}(\delta)-40\gamma=0$;

\item[{\rm(iv)}]$5\mid q$, $\Tr_{q}^{q^2}(\delta)\ne0$ and
$\gamma+2\Tr_{q}^{q^2}(\delta)=0$.
\end{enumerate}
\end{thm}

\begin{thm}\label{Appx.2}
Let $q$ be an odd prime power, $\delta\in\mathbb{F}_{q^2}$ and $\gamma\in\mathbb{F}_q^*$. Then the polynomial
$$f(x) = (x^q-x+\delta)^{2q+4}+(x^q-x+\delta)^{q+5}+\gamma x$$
is a PP of $\mathbb{F}_{q^2}$ if and only if one of the following holds:
\begin{enumerate}
\item[{\rm (i)}]$\Tr_{q}^{q^2}(\delta)=0$;
\item[{\rm (ii)}]
    $\Tr_{q}^{q^2}(\delta)\neq0$,  $q\equiv\pm2\pmod5$ and $40\gamma=19\left(\Tr_{q}^{q^2}(\delta)\right)^5$;
\item[{\rm (iii)}]
    $\Tr_{q}^{q^2}(\delta)\neq0$,    $5\mid q$ and $2\gamma=\left(\Tr_{q}^{q^2}(\delta)\right)^5$.
\end{enumerate}
\end{thm}

\begin{thm}\label{Appx.3}
Let $q$ be an odd prime power, $\delta\in\mathbb{F}_{q^2}$ and $\gamma\in\mathbb{F}_q^*$. Then the polynomial
$$f(x) = (x^q-x+\delta)^{2q+4}+(x^q-x+\delta)^{q}+\gamma x$$
is a PP of $\mathbb{F}_{q^2}$ if and only if  one of the following holds:
 \begin{enumerate}
\item[{\rm (i)}] $\Tr_{q}^{q^2}(\delta)=0$ and $\gamma+2\ne0$;
\item[{\rm (ii)}]$\Tr_{q}^{q^2}(\delta)\ne0$, $q\equiv\pm2\pmod5$ and $5\gamma=\left(\Tr_{q}^{q^2}(\delta)\right)^2-10$;
\item[{\rm (iii)}]$\Tr_{q}^{q^2}(\delta)\ne0$, $5\mid q$ and $\gamma=3$.
\end{enumerate}
\end{thm}

\begin{thm}\label{Appx.4}
Let $q$ be an odd prime power, $\delta\in\mathbb{F}_{q^2}$ and $\gamma\in\mathbb{F}_q^*$. Then the polynomial
$$f(x) = (x^q-x+\delta)^{2q+3}+(x^q-x+\delta)^{5q}+\gamma x$$
is a PP of $\mathbb{F}_{q^2}$ if and only if one of the following holds:
 \begin{enumerate}
\item[{\rm (i)}]  $\Tr_{q}^{q^2}(\delta)=0$;
\item[{\rm (ii)}]$\Tr_{q}^{q^2}(\delta)\ne0$, $3\mid q$ and $\gamma\neq \left(\Tr_{q}^{q^2}(\delta)\right)^4$;
\item[{\rm (iii)}]$\Tr_{q}^{q^2}(\delta)\ne0$,  $q\equiv 2\pmod 3$ and $ \gamma=-\left(\Tr_{q}^{q^2}(\delta)\right)^4/2$.
\end{enumerate}
\end{thm}

\begin{thm}\label{Appx.5}
Let $q$ be an odd prime power, $\delta\in\mathbb{F}_{q^2}$ and $\gamma\in\mathbb{F}_q^*$. Then the polynomial
$$f(x) = (x^q-x+\delta)^{2q+4}+(x^q-x+\delta)^{2q}+\gamma x, $$
is a PP of $\mathbb{F}_{q^2}$ if and only if one of the following holds:
\begin{enumerate}
\item[{\rm (i)}]  $\Tr_{q}^{q^2}(\delta)=0$;
\item[{\rm (ii)}]$\Tr_{q}^{q^2}(\delta)\neq0$, $q\equiv\pm2\pmod5$ and $\gamma=\left(\Tr_{q}^{q^2}(\delta)\right)^5/40-2\Tr_{q}^{q^2}(\delta)$;
\item[{\rm (iii)}]$\Tr_{q}^{q^2}(\delta)\neq0$,
    $5\mid q$ and $\gamma=-2\Tr_{q}^{q^2}(\delta)$.
\end{enumerate}
\end{thm}


\begin{thm}\label{Appx.6}
Let $q$ be an odd prime power, $\delta\in\mathbb{F}_{q^2}$ and $\gamma\in\mathbb{F}_q^*$. Then the polynomial
$$f(x) = (x^q-x+\delta)^{ q+5}+(x^q-x+\delta)^{2 q}+\gamma x, $$
is a PP of $\mathbb{F}_{q^2}$ if and only if one of the following holds:
\begin{enumerate}
\item[{\rm (i)}]  $\Tr_{q}^{q^2}(\delta)=0$;

\item[{\rm (ii)}] $\Tr_{q}^{q^2}(\delta)\neq0$, $q\not\equiv1\pmod5$ and $\gamma=\left(\Tr_{q}^{q^2}(\delta)\right)^5/4-2\Tr_{q}^{q^2}(\delta)$;
\item[{\rm (iii)}] $\Tr_{q}^{q^2}(\delta)\neq0$, $5\mid q$ and
    $\left(\left(\Tr_{q}^{q^2}(\delta)\right)^5+2\Tr_{q}^{q^2}(\delta)+\gamma\right)\big/\Tr_{q}^{q^2}(\delta)$ is a fourth power or not a square in $\mathbb{F}_q$;

\item[{\rm (iv)}] $\Tr_{q}^{q^2}(\delta)\neq0$, $q=9$ and  $\gamma= \left(\Tr_{q}^{q^2}(\delta)\right)^5$, or $\gamma=
    \left(\Tr_{q}^{q^2}(\delta)\right)^5-\Tr_{q}^{q^2}(\delta)$.
\end{enumerate}
 \end{thm}

{\bf Remark:} For the case $q=9$ and $\Tr_{q}^{q^2}(\delta)\neq0$ in Theorem \ref{Appx.6}, we can explicitly determine all possible pairs $(\Tr_{q}^{q^2}(\delta), \gamma)$. Since the polynomial $x^2+1$ is irreducible over $\mathbb{F}_3$, we have $\mathbb{F}_9\cong\mathbb{F}_3[x]/(x^2+1)$, with elements represented as $c+di$ with $c,d\in\mathbb{F}_3$, and $i^2=-1$. Let $g=1+i\in\mathbb{F}_9$ be a generator of $\mathbb{F}_9^*$. Then every non-zero trace value  can be written as $\Tr_{q}^{q^2}(\delta)=g^k$ for $0\leq k\leq 7$. The corresponding values of $\gamma$ are given by $\gamma=g^{5k}$ or $\gamma=g^{5k}-g^k$. A direct computation yields exactly 12 distinct admissible pairs  $(\Tr_{q}^{q^2}(\delta), \gamma)$. More precisely, $(\Tr_{q}^{q^2}(\delta), \gamma)\in\{(1,1), (1+i,2+2i), (2i,2i),(1+2i,2+i),(2,2),(2+2i,1+i),(i,i),(i+2,1+2i),(1+i,1+i),(1+2i,1+2i),(2+2i,2+2i),(2+i,2+i)\}$.

\medskip

 All PPs presented in Theorems \ref{Wuyuan}-\ref{Appx.6} are of the form $(2)$ or $(3)$ with $L(x)=x$. Recently, Liu, Jiang and Zou \cite{Liu-Jiang-Zou2025} studied PPs with the form $\sum_i(x^{p^m}-x+\delta)^{s_i}+\gamma(x^{p^m}+x)$ over $\mathbb{F}_{p^{2m}}$, where $\gamma\in\mathbb{F}_{p^m}^*$ and $p\in\{3,5\}$, and proposed six new classes of  such PPs. In what follows, we aim to generalize their results to finite fields with arbitrary odd characteristic.

\begin{thm}\label{Liu-Jiang-Zou1}
Let $p$ be an odd prime and $q=p^m$. For a positive integer $i<m$, $\delta\in\mathbb{F}_{q^2}$  and $\gamma\in\mathbb{F}_{q}^\ast$, the polynomial
$$f(x)=(x^q-x+\delta)^{q+p^i}+\gamma(x^q+x), $$
is a PP of $\mathbb{F}_{q^2}$ if and only if $\Tr_q^{q^2}(\delta)\ne0$.
\end{thm}
\begin{proof} Since $q$ is an odd prime power,  there is a non-square element $u\in\mathbb{F}_{q}$. Let $\alpha$ be an element of $\mathbb{F}_{q^2}$ with $\alpha^2=u$. Then $\{1, \alpha\}$ is a basis of $\mathbb{F}_{q^2}$ over $\mathbb{F}_q$. Let $\delta=a+b\alpha$ with  $a, b\in\mathbb{F}_q$ and $x=y-(z-b)\alpha/2$. Then $2a=\Tr_q^{q^2}(\delta)$ and $x^q-x+\delta=a+z\alpha$. We have
\begin{align*}
f(x)=f(y-(z-b)\alpha/2)&=(a-z\alpha)(a+z\alpha)^{p^i}+2\gamma y\\
&=2\gamma y-z^{p^i+1}u^{(p^i+1)/2}+a^{p^i+1}+a(u^{(p^i-1)/2}z^{p^i}-a^{p^i-1}z)\alpha.
\end{align*}
Hence
\begin{equation*}
\begin{cases}
g_1(y, z)=2\gamma y-z^{p^i+1}u^{(p^i+1)/2},\\
 g_2(y,z)=a(u^{(p^i-1)/2}z^{p^i}-a^{p^i-1}z).
 \end{cases}
 \end{equation*}
Observe that $g_2(y, z)$ is a polynomial of $z$ and the $y$-part of $g_1(y, z)$ is $2\gamma y$. Hence, $(g_1(y,z),g_2(y,z))$ permutes $\mathbb{F}_q^2$ if and only if $g_2(y, z)=g_2(z)$ is a PP over $\mathbb{F}_q$. Clearly, if $a=0$, then $g_2(z)=0$ cannot be a PP of $\mathbb{F}_q$. Now assume $a\neq0$.
Clearly $g_2(z)$ is a linearized polynomial over $\mathbb{F}_q$. Therefore, $g_2(z)$ is a PP of $\mathbb{F}_q$ if and only if $$u^{(p^i-1)/2}z^{p^i-1}-a^{p^i-1}=0$$
 has no solution in  $\mathbb{F}_q$, which is equivalent to $a=\Tr_q^{q^2}(\delta)/2\ne0$, since $u$ is a non-square in $\mathbb{F}_q$. It follows that $f(x)$ is a PP of $\mathbb{F}_{q^2}$ if and only if $\Tr_q^{q^2}(\delta)\ne0$. Thus the theorem is established.
\end{proof}

\begin{thm}\label{Liu-Jiang-Zou2}
Let $q$ be an odd prime power, $\delta\in\mathbb{F}_{q^2}$ and $\gamma\in\mathbb{F}_{q}^\ast$. Then the polynomial
$$f(x)=(x^q-x+\delta)^{q+2}+\gamma(x^q+x), $$
is a PP of $\mathbb{F}_{q^2}$ if and only if
$q\equiv0\pmod3$, or  $q\equiv2\pmod{3}$ and $\Tr_q^{q^2}(\delta)=0$.
\end{thm}


\begin{thm}\label{Liu-Jiang-Zou3}
Let $q$ be an odd prime power, $\delta\in\mathbb{F}_{q^2}$ and $\gamma\in\mathbb{F}_{q}^\ast$. Then the polynomial
$$f(x)=(x^q-x+\delta)^{3 q+2}+\gamma(x^q+x), $$
is a PP of $\mathbb{F}_{q^2}$ if and only if
 $q\equiv0\pmod 5$, or $q\equiv2,3,4\pmod5$ and $\Tr_q^{q^2}(\delta)=0$.
\end{thm}


\begin{thm}\label{Liu-Jiang-Zou4}
Let $q$ be an odd prime power, $\delta\in\mathbb{F}_{q^2}$ and $\gamma\in\mathbb{F}_{q}^\ast$. Then the polynomial
$$f(x)=(x^q-x+\delta)^{4 q+2}+\gamma(x^q+x),$$
is a PP of $\mathbb{F}_{q^2}$ if and only if $q\equiv0\pmod5$ and $\Tr_q^{q^2}(\delta)\neq0$.
\end{thm}

 {\bf Remark:}  For $q=3^m$, Theorems \ref{Liu-Jiang-Zou1} and \ref{Liu-Jiang-Zou2} in this work reduce to Theorem 1 and 2 in \cite{Liu-Jiang-Zou2025}, respectively. For $q=5^m$, Theorems \ref{Liu-Jiang-Zou3} and \ref{Liu-Jiang-Zou4} reduce to Theorem 3 and 4 in \cite{Liu-Jiang-Zou2025}. Due to the similarity in the argument, detailed proofs of Theorems \ref{Liu-Jiang-Zou2}-\ref{Liu-Jiang-Zou4} are omitted here.

\begin{thm}\label{Liu-Jiang-Zou5}
Let $q$ be an odd prime power,  $\delta\in\mathbb{F}_{q^2}$ and $\gamma\in\mathbb{F}_{q}^\ast$. Then the polynomial
$$f(x)=(x^q-x+\delta)^{ q+3}+(x^q-x+\delta)^{ q+2}+\gamma(x^q+x),$$
is a PP of $\mathbb{F}_{q^2}$ if and only if
 $q\equiv0\pmod3$ and $\Tr_q^{q^2}(\delta)\neq2$, or $q\equiv2\pmod3$ and $\Tr_q^{q^2}(\delta)=0$.
\end{thm}
\begin{proof}
Since $q$ is an odd prime power,  there is a non-square element $u\in\mathbb{F}_{q}$. Let $\alpha$ be an element of $\mathbb{F}_{q^2}$ with $\alpha^2=u$. Then $\{1, \alpha\}$ is a basis of $\mathbb{F}_{q^2}$ over $\mathbb{F}_q$. Let $\delta=a+b\alpha$ with $a, b\in\mathbb{F}_q$ and $x=y-(z-b)\alpha/2$. Then $2a=\Tr_q^{q^2}(\delta)$ and $x^q-x+\delta=a+z\alpha$. We have
\begin{align*}
f(x)=f(y-(z-b)\alpha/2)&=(a^2-z^2u)(a+z\alpha)^2+(a^2-z^2u)(a+z\alpha)+2\gamma y\\
   &=2\gamma y-u^2z^4-auz^2+a^4+a^3+(2a+1)(-uz^3+a^2z)\alpha.
\end{align*}
Hence
\begin{equation*}
\begin{cases}
g_1(y, z)=2\gamma y-u^2z^4-auz^2,\\
g_2(y, z)=(2a+1)(-uz^3+a^2z).
\end{cases}
\end{equation*}
Observe that $g_2(y, z)$ is a polynomial of $z$ and the $y$-part of $g_1(y, z)$ is $2\gamma y$. Hence $(g_1(y,z),g_2(y,z))$ permutes $\mathbb{F}_q^2$ if and only if $g_2(y, z)=g_2(z)$ is a PP over $\mathbb{F}_q$. If $2a+1=0$, $g_2(z)=0$ for all $z\in\mathbb{F}_q$, which is clearly not a PP. We therefore assume that $2a+1\neq0$. Then the normalized form of $g_2(z)$ is as follows
$$z^3-\frac{a^2}{u}z.$$
From Table 7.1 in \cite{LN97}, the polynomial $g_2(z)$ is a PP over $\mathbb{F}_q$ if and only if $q\not\equiv1\pmod{3}$ and
 $a=0$, or $3\mid q$ and $a\neq0$. Combining with $2a+1\ne0$, we obtain that $f(x)$ is a PP of $\mathbb{F}_{q^2}$ if and only if $q\not\equiv1\pmod{3}$ and $\Tr_q^{q^2}(\delta)=0$, or $3\mid q$ and
 $\Tr_q^{q^2}(\delta)\ne0$, $\Tr_q^{q^2}(\delta)+1\ne0$ (i.e,  $\Tr_q^{q^2}(\delta)\not\in\{0,2\}$). We complete the proof.\end{proof}


\begin{thm}\label{Liu-Jiang-Zou6}
Let $q$ be an odd prime power, $\delta\in\mathbb{F}_{q^2}$ and $\gamma\in\mathbb{F}_{q}^\ast$. Then the polynomial
$$f(x)=(x^q-x+\delta)^{ 3 q+2}+(x^q-x+\delta)^{4 q+2}+\gamma(x^q+x), $$
is a PP of $\mathbb{F}_{q^2}$ if and only if
$q\equiv0\pmod5$ and $\Tr_q^{q^2}(\delta)\neq-1$, or $q\equiv2,3,4\pmod5$ and $\Tr_q^{q^2}(\delta)=0$.
\end{thm}

 {\bf Remark:}  When $q=5^m$, Theorems \ref{Liu-Jiang-Zou5} and \ref{Liu-Jiang-Zou6} yield Theorems 5 and 6 in \cite{Liu-Jiang-Zou2025}, respectively. The proof of Theorem \ref{Liu-Jiang-Zou6} has been omitted herein.

\subsection{Even characteristic}

By the end of this section, we propose a class of PPs over $\mathbb{F}_{2^{2m}}$, thereby demonstrating that our method remains valid in the case of even characteristic.
Xu, Luo and Cao gave the necessary and sufficient conditions for $f(x)=(x^{2^m}+x+\delta)^{2^{m+1}+1}+x$ over $\mathbb{F}_{2^{2m}}$ to be PPs in \cite[Theorem 3.2]{XuLC22}. We now extend their result to the following general form.

\begin{thm}
 Let $m$ be a positive integer and $q=2^m$.  For $\delta,\gamma\in\mathbb{F}_{q^2}$ with $\gamma\neq0$,  the polynomial
$$f(x)=(x^{q}+x+\delta)^{2q+1}+\gamma x, $$
is a PP of $\mathbb{F}_{q^2}$ if and only if one of the following holds:
\begin{enumerate}
\item[{\rm (i)}]  $\gamma\in\mathbb{F}_{q}^*$,  and $b=0$ or $b^2=\gamma$;

\item[{\rm (ii)}]  $\gamma\in\mathbb{F}_{q^2}^*\backslash\mathbb{F}_{q}$, $c=b^2u+b^2+du=0$ and $m$ is odd;
\item[{\rm (iii)}] $\gamma\in\mathbb{F}_{q^2}^*\backslash\mathbb{F}_{q}$,
$c\neq0$, $b^2c^2+d\left(b^2(c+d+du)+c^2+cd+d^2u\right)=0$ and $m$ is odd;
\end{enumerate}
where $u\in\mathbb{F}_q$ satisfies  $\Tr_2^{q}(u)=1$, and $a,b,c,d\in\mathbb{F}_q$ satisfy $\Tr_{q}^{q^2}(\delta)=b$, $\Tr_{q}^{q^2}(\gamma)=d$, $\N_{q}^{q^2}(\delta)=a^2+ab+b^2u$ and $\N_{q}^{q^2}(\gamma)=c^2+cd+d^2u$.
\end{thm}
\begin{proof}
Let $u\in\mathbb{F}_{q}$ such that $\Tr_2^{q}(u)=1$. It follows that the polynomial $x^2+x+u$ is irreducible over $\mathbb{F}_{q}$. Now, let  $\alpha\in\mathbb{F}_{q^2}$ be a root of this polynomial. Since the polynomial is irreducible and of degree 2, the extension $\mathbb{F}_{q}(\alpha)$ has degree 2 over $\mathbb{F}_{q}$, and is thus equal  to $\mathbb{F}_{q^2}$. Consequently, $\{1,\alpha\}$ is a basis of $\mathbb{F}_{q^2}$ over $\mathbb{F}_{q}$.
Let $\delta=a+b\alpha$, $\gamma=c+d\alpha$ and $x=y+(z+a)\alpha$ with $a,b,c,d,y,z\in\mathbb{F}_{q}$. A direct calculation gives that $\alpha^{q}=1+\alpha$ and thus $x^q+x+\delta=z+b\alpha$. Hence, we have
\begin{align*}
f(x)=f(y+(z+a)\alpha)&=(z+b\alpha)^{2q+1}+(c+d\alpha)(y+(z+a)\alpha)\\
&=(z+b\alpha)(z+b(1+\alpha))^2+(c+d\alpha)(y+(z+a)\alpha)\\
&=cy+z^3+(b^2u+b^2+du)z+b^3u+adu\\
   &~~~~+\left(dy+bz^2+(b^2+c+d)z+b^3u+ac+ad\right)\alpha.
\end{align*}
Next it is sufficient to consider the permutation properties of $(g_1(y,z),g_2(y,z))$ listed below.
\begin{equation*}
\begin{cases}
g_1(y,z)=cy+z^3+(b^2u+b^2+du)z,\\
g_2(y,z)=dy+bz^2+(b^2+c+d)z.
\end{cases}
\end{equation*}
Since $\gamma\in\mathbb{F}_{q^2}^*$, it follows that $c$ and $d$ cannot both be zero. We therefore divide the analysis into three cases.

Case 1: Assume $c\neq0$ and $d=0$.  Observe that $g_2(y,z)=g_2(z)$ is a polynomial of $z$ and the $y$-part of $g_1(y,z)$ is $cy$. Hence $(g_1(y,z),g_2(y,z))$ permutes $\mathbb{F}_{q}^2$ if and only if $g_2(z)$ permutes $\mathbb{F}_{q}$. Note that
$g_2(z)=bz^2+(b^2+c)z$ permutes $\mathbb{F}_{q}$ if and only if $b=0$ or $b^2+c=0$.  It follows that $f(x)$ permutes $\mathbb{F}_{q^2}$ if and only if $b=0$ or $b^2= \gamma$.

Case 2: Assume $c=0$ and $d\neq0$. Then $g_1(y,z)=g_1(z)=z^3+(b^2u+b^2+du)z$, which depends only on $z$. By Table 7.1 in \cite{LN97}, $g_1(z)$ is a PP of $\mathbb{F}_{q}$ if and only if $q\not\equiv1\pmod3$ and $b^2u+b^2+du=0$ . Note that the $y$-part of $g_2(y,z)$ is $dy$. Therefore,  $(g_1(y,z),g_2(y,z))$ permutes $\mathbb{F}_{q}^2$
 if and only if $g_1(z)$ permutes $\mathbb{F}_{q}$, which in turn is equivalent to $q\not\equiv1\pmod3$ and $b^2u+b^2+du=0$. So $f(x)$ permutes $\mathbb{F}_{q^2}$ if and only if  $m$ is odd and $b^2u+b^2+du=0$.

Case 3: Assume $cd\neq0$. Observe the $y$-part of $g_1(y,z)$ and $g_2(y,z)$ are $cy$ and $dy$, respectively. Then $(g_1(y,z),g_2(y,z))$ permutes $\mathbb{F}_q^2$ if and only if
\begin{align*}
d^{-1}h(z)&=d^{-1}\left(d\cdot g_1(y,z)+c\cdot g_2(y,z)\right)\\
   &=d^{-1}\left(dz^3+bcz^2+b^2(c+d+du)z+(c^2+cd+d^2u)z\right)\\
   &=z^3+\frac{bc}{d}z^2+\frac{b^2(c+d+du)+c^2+cd+d^2u}{d}z
\end{align*}
permutes $\mathbb{F}_q$. By Table 7.1 in \cite{LN97}, the normalized form of $d^{-1}h(z)$ is
$$z^3+\left(\frac{b^2c^2}{d^2}+\frac{b^2(c+d+du)+c^2+cd+d^2u}{d}\right)z,$$
which induces a permutation of $\mathbb{F}_q$ if and only if $q\not\equiv1\pmod3$ and $\frac{b^2c^2}{d^2}+\frac{b^2(c+d+du)+c^2+cd+d^2u}{d}=0$ . That is, $m$ is odd and $b^2c^2+d\left(b^2(c+d+du)+c^2+cd+d^2u\right)=0$.
\end{proof}

 {\bf Remark:} Indeed, when $d=0$ (i.e., $\gamma=c\in\mathbb{F}_{q}$), the same result can also be derived using the commutative diagram below:
\begin{center}
 \quad
\xymatrix{
  \mathbb{F}_{q^2} \ar[d]_{\Tr_q^{q^2}(x)} \ar[r]^{f} & \mathbb{F}_{q^2} \ar[d]_{\Tr_q^{q^2}(x)}  \\
  \mathbb{F}_{q} \ar[r]^{h(x)} & \mathbb{F}_{q}   }\end{center}
where $h(x)=\Tr_{q}^{q^2}(\delta)x^2+\left(\left(\Tr_{q}^{q^2}(\delta)\right)^2+c\right)x+\delta^{q+1}\Tr_{q}^{q^2}(\delta)$, which permutes $\mathbb{F}_{{2^m}}$ if and only if $h'(x)=\Tr_{q}^{q^2}(\delta)x^2+\left(\left(\Tr_{q}^{{q^2}}(\delta)\right)^2+c\right)x$ permutes $\mathbb{F}_{2^m}$.

%
%
%

\section{Another method for constructing PPs over  $\mathbb{F}_{q^d}$}
At this point, it would seem appropriate to further construct some classes of PPs over $\mathbb{F}_{q^d}$ for $d\geq3$ using our approach. However, we defer this task to a future paper, where we will find a more optimal basis for the extension $\mathbb{F}_{q^d}/\mathbb{F}_q$. This section notes that for polynomials with certain special forms, other methods may also be applicable.

Recently, Jiang, Li and Qu  investigated PPs of the form $x+\gamma\Tr_q^{q^3}(h(x))$  over finite fields with even characteristic in \cite{LiKQ2026}. Inspired by their work,  we recover and extend their Theorem 3.3, as stated below.

\begin{thm}
Let $q=2^m$ with $m>1$. Let $d$ be an odd positive integer and  $\gamma\in\mathbb{F}_{q^d}$. Then the polynomial
$$f(x)=x+\gamma\Tr_{q}^{q^d}(x^{q+1}+x^{2q+2}),$$
is a PP of $\mathbb{F}_{q^d}$ if and only if $\gamma \in \mathbb{F}_q$ and  $\gamma t^3+\gamma t+1=0$ has no solution in $\mathbb{F}_q$.
\end{thm}

\begin{proof}
By definition, a polynomial $f(x)$ permutes $\mathbb{F}_q$ if and only if, for every $a\in\mathbb{F}_{q^d}^*$, $f(x+a)+f(x)=0$ has no solution in $\mathbb{F}_{q^d}.$ We note that for $\gamma=0$, $f(x)=x$ is clearly a PP over $\mathbb{F}_{q^d}$. Now assume $\gamma\neq0$. Then we have
$$f(x+a)+f(x)=a+\gamma\Tr_q^{q^d}((x+a)^{q+1}+(x+a)^{2q+2})+\gamma\Tr_{q}^{q^d}(x^{q+1}+x^{2q+2})=0,$$
or, equivalently,
\begin{equation*}
\left(\Tr_q^{q^d}((a^q+a^{q^{d-1}})x)\right)^2+\Tr_q^{q^d}((a^q+a^{q^{d-1}})x)+\Tr_q^{q^d}(a^{q+1}+a^{2q+2})+\frac{a}{\gamma}=0.\tag{5}\label{eq5}
\end{equation*}
This implies that $\frac{a}{\gamma}\in\mathbb{F}_q$. If $\gamma\in\mathbb{F}_q^*$, then  $a\in\mathbb{F}_q$. So equation (\ref{eq5}) becomes $a^4+a^2+\frac{a}{\gamma}=0$ since $2\nmid d$. Consequently, equation (\ref{eq5})  has no solution in $\mathbb{F}_{q^d}$ if and only if $\gamma a^3+\gamma a+1=0$ has no solution in $\mathbb{F}_{q}$.  If $\gamma\in\mathbb{F}_{q^d}^*\backslash\mathbb{F}_q$, it follows immediately that $a\not\in\mathbb{F}_q$. Let $y=\Tr_q^{q^d}((a^q+a^{q^{d-1}})x)$. We now claim that $y\not\equiv0$. Then equation (\ref{eq5}) has no solution in $\mathbb{F}_{q^d}$ if and only if
\begin{equation*}
y^2+y+\Tr_q^{q^d}(a^{q+1})+\left(\Tr_q^{q^d}(a^{q+1})\right)^2+\frac{a}{\gamma}=0\tag{6}\label{eq6}
\end{equation*}
has no solution in $\mathbb{F}_q$.  Choose $b\in\mathbb{F}_q^*$ such that $a=(b^2+b)\gamma$. Substituting  into \eqref{eq6} and simplifying,  one may verify that $y=\Tr_q^{q^d}(a^{q+1})+b$ satisfies the equation. Therefore, there exist $a\in\mathbb{F}_{q^d}^*$ such that (\ref{eq6}) has a solution in $\mathbb{F}_q$, and thus (5) also admits a solution in $\mathbb{F}_{q^d}$. This contradicts the permutation property of $f(x)$. Hence, for each $\gamma\in\mathbb{F}_{q^d}^*\backslash\mathbb{F}_q$, $f(x)$ cannot be a PP over $\mathbb{F}_{q^d}$.

{\em proof of the claim:} To prove  $y\not\equiv0$, it suffices to show
$a^q+a^{q^{d-1}}\neq0$. Assume for contradiction that $a^q+a^{q^{d-1}}=0$. Raising both sides to the $q$-th power gives $a^{q^2}+a=0$, so $a^{q^2-1}=1$. On the other hand, we have $a^{q^d-1}=1$ as $a\in\mathbb{F}_{q^d}^*$. Using the identity
$$\gcd(q^2-1,q^d-1)=q^{\gcd(2,d)}-1=q-1,$$
we obtain $a^{q-1}=1$, and thus $a\in\mathbb{F}_q$, a contradiction.
\end{proof}


\section{Summary and concluding remarks}
Undoubtedly, the construction of permutation polynomials over finite fields stands a major research focus in the algebraic domain. In recent years, a large number of  permutation polynomials have been constructed using the AGW criterion \cite{AGW}, a powerful and widely applicable method. In this paper, we dedicate to applying some new  approaches to constructing permutation polynomials. In Section 2,  we demonstrate that the permutation behavior of a univariate polynomial can be translated to that of its corresponding multivariate polynomial system, and vice versa. As one application of our method, we propose a necessary and sufficient condition for polynomials of the form $x+g(\Tr_{q}^{q^n}(x))$ over finite fields to be  permutations at the end of this section.

Note that permutation polynomials of the form $\sum_i(x^q-x+\delta)^{s_i}+L(x)$ have been extensively studied. In Section 3, we extend some of known results using our proposed approach. Many of the findings presented in this paper are therefore generalizations of earlier results on permutation polynomials, as referenced throughout. For the reader's convenience, Table 1 offers a systematic summary of permutation polynomials over $\mathbb{F}_{q^2}$ in the case of odd characteristic. In Section 4, we  present an alternative approach for constructing permutation polynomials with certain special form.

We believe that more new permutation polynomials can be obtained through the ideas presented in this paper. While the present work primarily addresses polynomials over $\mathbb{F}_{q^2}$, the approach can be naturally extended to higher-degree extensions. Notably, finding an optimal basis is vital for this generalization. Our future work will therefore be dedicated to studying permutation polynomials over  higher extensions.


Finally, we would like to emphasize that  studying the polynomials of the form $f(x)+\gamma x$ is highly meaningful, as such forms are closely connected to the concept of direction sets \cite{Ball,Blokhuis-Ball-Brouwer-Storme-Szonyi}. For a mapping $f :\mathbb{F}_{q^n}\to \mathbb{F}_{q^n}$,  its  direction set is  defined as $$D(f) :=\left\{ \frac{ f(x)-f(y)}{x-y}| x, y\in\mathbb{F}_q, \,\, x\ne y\right\}.$$
One may observe that determining the size of $D(f)$  is equivalent to characterizing the set
$$P(f):= \{ \gamma\in\mathbb{F}_{q^n}|  f(x) + \gamma x \,\, \mbox{ is a permutation of}\,\,\mathbb{F}_{q^n}\}.$$
More precisely: If $\left(f(x)-f(y)\right)/(x-y)=m$, then $f(x)-mx=f(y)-my$, so $m$ is a direction determined by $f(x)$ precisely when $f(x)-mx$ is not a permutation polynomial over $\mathbb{F}_{q^n}$. In other words, $m\in D(f)$ if and only if $-m\not\in P(f)$.

\bigskip

\begin{table}[htbp]
  \centering
  \caption{PPs  of the form $\sum_i(x^q-x+\delta)^{s_i}+\gamma L(x)$ \\
  over $\mathbb{F}_{q^2}$ with odd characteristic in this paper}\label{table1}
  \setlength{\tabcolsep}{3pt}    
  \renewcommand{\arraystretch}{0.9}  
  \begin{tabular}{c|cc}
    \toprule
  $\gamma$ & $f(x)$ & {\rm Rf} \\
  \hline
 \multirow{5}{*}{$\gamma\in\mathbb{F}_{q^2}^*$ } &$(x^q-x+\delta)^{q+2}+\gamma x$ &Theorem \ref{Wuyuan}  \\
 &$(x^q-x+\delta)^{2}+\gamma x$   &Theorem \ref{Li_2} \\
  &$(x^q-x+\delta)^{2q}+\gamma x$   &Theorem \ref{Li_2q} \\
 &$(x^q-x+\delta)^{q+2}+(x^q-x+\delta)^{2q+1}+\gamma x$    &Theorem \ref{q+2_2q+1} \\
&$(x^q-x+\delta)^{q+4}+(x^q-x+\delta)^{5}+\gamma x$  &Theorem \ref{degree5}\\
  \hline
\multirow{13}{*}{$\gamma\in\mathbb{F}_{q}^*$ } &$(x^q-x+\delta)^{2q+1}+(x^q-x+\delta)^{3q+2}+\gamma x$ & Theorem \ref{Licao2023th}\\
 & $(x^q-x+\delta)^{2q+3}+(x^q-x+\delta)^{2q}+\gamma x$ &Theorem \ref{Appx.1}\\
 &$(x^q-x+\delta)^{2q+4}+(x^q-x+\delta)^{q+5}+\gamma x$  &Theorem \ref{Appx.2} \\
& $(x^q-x+\delta)^{2q+4}+(x^q-x+\delta)^{q}+\gamma x$ &Theorem \ref{Appx.3}\\
 & $(x^q-x+\delta)^{2q+3}+(x^q-x+\delta)^{5q}+\gamma x$ &Theorem \ref{Appx.4}\\
& $(x^q-x+\delta)^{2q+4}+(x^q-x+\delta)^{2q}+\gamma x$  &Theorem \ref{Appx.5}\\
 & $(x^q-x+\delta)^{ q+5}+(x^q-x+\delta)^{2q}+\gamma x$  &Theorem \ref{Appx.6}\\
 &$(x^q-x+\delta)^{ q+p^i}+\gamma(x^q+x)$ &Theorem \ref{Liu-Jiang-Zou1}\\
 &$(x^q-x+\delta)^{ q+2}+\gamma(x^q+x)$ &Theorem \ref{Liu-Jiang-Zou2} \\
 &$(x^q-x+\delta)^{ 3q+2}+\gamma(x^q+x)$ &Theorem \ref{Liu-Jiang-Zou3}\\
 & $(x^q-x+\delta)^{ 4q+2}+\gamma(x^q+x)$ &Theorem \ref{Liu-Jiang-Zou4} \\
 & $(x^q-x+\delta)^{q+3}+(x^q-x+\delta)^{q+2}+\gamma(x^q+x)$ &Theorem \ref{Liu-Jiang-Zou5}\\
 & $(x^q-x+\delta)^{3q+2}+(x^q-x+\delta)^{4q+2}+\gamma(x^q+x)$ &Theorem \ref{Liu-Jiang-Zou6}\\
\bottomrule
\end{tabular}
\end{table}

\section*{Declarations}
\begin{conflict of interest} {\rm There is no conflict of interest.}
\end{conflict of interest}

\section*{Appendix}

 {\bf Proof of Theorem \ref{Appx.1}~} Since $q$ is an odd prime power,  there is a non-square element $u\in\mathbb{F}_{q}$. Let $\alpha$ be an element of $\mathbb{F}_{q^2}$ with $\alpha^2=u$. Then $\{1, \alpha\}$ forms a basis of $\mathbb{F}_{q^2}$ over $\mathbb{F}_q$. Let $\delta=a+b\alpha$ with $a, b\in\mathbb{F}_q$ and $x=y-(z-b)\alpha/2$. A simple computation shows that $2a=\Tr_q^{q^2}(\delta)$ and $x^q-x+\delta=a+z\alpha$. Then we have
\begin{align*}
f(x)&= (x^q-x+\delta)^{2q+3}+(x^q-x+\delta)^{2q}+\gamma x\\
&=(a^2-z^2u)^2(a+z\alpha)+(a-z\alpha)^2+\gamma(y-(z-b)\alpha/2)\\
 &=\gamma y +au^2z^4 + ( 1-2a^3 )uz^2 + a^2+ a^5 +
\left( u^2 z^5 - 2u a^2z^3 + (a^4 - 2a - \gamma/2)z + \gamma b/2\right)\alpha.
\end{align*}
Hence
\begin{equation*}
\begin{cases}
g_1(y, z)=\gamma y +au^2z^4 + ( 1-2a^3 )uz^2,\\
g_2(y, z)= u^2 z^5 - 2u a^2z^3 + (a^4 - 2a - \frac{\gamma }{2})z.
\end{cases}
\end{equation*}
Observe that $g_2(y, z)$ is a polynomial of $z$ and the $y$-part of $g_1(y, z)$ is $\gamma y$, which is a PP over $\mathbb{F}_q$. Therefore $(g_1(y,z),g_2(y,z))$ permutes $\mathbb{F}_q^2$ if and only if $g_2(y, z)=g_2(z)$ is a PP over $\mathbb{F}_q$.
Since $u$ is a non-square, the normalized form of $g_2(z)$ is as follows
$$z^5-\frac{2a^2}{ u}z^3+\frac{2a^4-4a-\gamma}{2 u^2}z.$$
Recall that $\gamma\neq0$.
It follows from Table 7.1 in \cite{LN97} that $g_2(z)$ is a PP over $\mathbb{F}_q$ if and only if one of the following occurs:
 \begin{enumerate}
\item [(i)] $5\mid q$, $a=0$ and $\gamma/2u^2$ is not a fourth power in $\mathbb{F}_q$, i.e., $\gamma/2$ is a fourth power or not a square in $\mathbb{F}_q$;
\item[(ii)] $q=9$, $a=0$ and $$\left(-\frac{\gamma}{2u^2}\right)^2=2,$$
which is equivalent to $\gamma^2=2u^4$. Using an argument analogous to that in Theorem \ref{Licao2023th}, we get $u^4=2$ and therefore $\gamma=\pm1$;
\item[(iii)] $q\equiv\pm2\pmod{5}$, $a\ne0$, and
 $$\left(-\frac{2a^2}{u}\right)^2=5\cdot\frac{2a^4-4a-\gamma}{2 u^2}.$$
 That is, $2a^4-20a-5\gamma=0$
\item[(iv)] $5\mid q$, $a\ne0$ and
$$\left(\frac{a^2}{u}\right)^2=\frac{2a^4-4a-\gamma}{ 2u^2}.$$
That is $\gamma+4a=0$.
\end{enumerate}
$\hfill\square$

 {\bf Proof of Theorem \ref{Appx.2}~} Since $q$ is an odd prime power, there exists a non-square element $u\in\mathbb{F}_{q}$. Let $\alpha$ be an element of $\mathbb{F}_{q^2}$ with $\alpha^2=u$. Then $\{1, \alpha\}$ is a basis of $\mathbb{F}_{q^2}$ over $\mathbb{F}_q$. Let $\delta=a+b\alpha$ with $a, b\in\mathbb{F}_q$ and $x=y-(z-b)\alpha/2$.  A simple computation shows that $2a=\Tr_q^{q^2}(\delta)$ and $x^q-x+\delta=a+z\alpha$. Then we have
\begin{align*}
f(x)&=(x^q-x+\delta)^{2q+4}+(x^q-x+\delta)^{q+5}+\gamma x\\
   &=(a^2-z^2u)^2(a+z\alpha)^2+(a^2-z^2u)(a+z\alpha)^4+\gamma(y-(z-b)\alpha/2)\\
  &=\gamma y-6a^2u^2z^4+4a^4uz^2+2a^6+\left(-2au^2z^5-4a^3uz^3+(6a^5-\frac{\gamma}{2})z+\frac{\gamma b}{2}\right)\alpha.
\end{align*}
Hence
\begin{equation*}
\begin{cases}
g_1(y, z)=\gamma y-6a^2u^2z^4+4a^4uz^2,\\
g_2(y, z)=-2au^2z^5-4a^3uz^3+(6a^5-\frac{\gamma}{2})z.
\end{cases}
\end{equation*}
%
Observe that $g_2(y, z)$ is a polynomial of $z$ and the $y$-part of $g_1(y, z)$ is $\gamma y$, which is a PP over $\mathbb{F}_q$. Therefore$(g_1(y,z),g_2(y,z))$ permutes $\mathbb{F}_q^2$ if and only if $g_2(y, z)=g_2(z)$ is a PP over $\mathbb{F}_q$. When $a=0$, $g_2(z)$ reduces to $g_2(z)=-\gamma z/2$, which is a PP of $\mathbb{F}_{q}$ as $\gamma\neq0$.  When $a\neq0$, the normalized form of $g_2(z)$ is
$$z^5+\frac{2a^2}{ u}z^3-\frac{12a^5-\gamma}{4a u^2}z.$$
In this case, by Table 7.1 in \cite{LN97}, the polynomial $g_2(z)$  is a PP over $\mathbb{F}_q$ if and only if  $q\equiv\pm2\pmod{5}$ and
$$\left(\frac{2a^2}{u}\right)^2=-5\cdot\frac{12a^5-\gamma}{4a u^2}.$$
That is $5\gamma=76a^5$, or $5\mid q$ and
$$\left(\frac{a^2}{u}\right)^2=-\frac{12a^5-\gamma}{4a u^2}.$$
That is $\gamma=16a^5\equiv a^5\pmod 5$.
$\hfill\square$

 {\bf Proof of Theorem \ref{Appx.3}~}
 Since $q$ is an odd prime power, there is a non-square element $u\in\mathbb{F}_{q}$. Let $\alpha$ be an element of $\mathbb{F}_{q^2}$ with $\alpha^2=u$. Then $\{1, \alpha\}$ is a basis of $\mathbb{F}_{q^2}$ over $\mathbb{F}_q$. Let $\delta=a+b\alpha$ with  $a, b\in\mathbb{F}_q$ and $x=y-(z-b)\alpha/2$. Then $2a=\Tr_{q}^{q^2}(\delta)$,
 $x^q-x+\delta=a+z\alpha$, and we have
\begin{align*}
f(x)&=(x^q-x+\delta)^{2q+4}+(x^q-x+\delta)^{q}+\gamma x \\
  &=(a^2-z^2u)^2(a+z\alpha)^2+(a-z\alpha)+\gamma(y-(z-b)\alpha/2)\\
  &=\gamma y+u^3z^6-a^2u^2z^4-a^4uz^2+a^6+a+\left(2au^2z^5-4a^3uz^3+(2a^5-1-\frac{\gamma}{2})z+\frac{b\gamma}{2}\right)\alpha.
\end{align*}
Hence
\begin{equation*}
\begin{cases}
g_1(y, z)=\gamma y+u^3z^6-a^2u^2z^4-a^4uz^2,\\
g_2(y, z)=2au^2z^5-4a^3uz^3+(2a^5-1-\frac{\gamma}{2})z.
\end{cases}
\end{equation*}
Observe that $g_2(y, z)$ is a polynomial of $z$ and the $y$-part of $g_1(y, z)$ is $\gamma y$. Therefore $(g_1(y,z),g_2(y,z))$ permutes $\mathbb{F}_q^2$ if and only if $g_2(y, z)=g_2(z)$ is a PP over $\mathbb{F}_q$. When $a=0$, $g_2(z)$ reduces to a linear polynomial, and thus induces a permutation of $\mathbb{F}_q$ only if $\gamma+2\neq0$. When $a\neq0$,
the normalized form of $g_2(z)$  is
$$z^5-\frac{2a^2}{ u}z^3+\frac{4a^5-2-\gamma}{4a u^2}z.$$
In this case, by Table 7.1 in \cite{LN97} , the polynomial $g_2(z)$ is a PP over $\mathbb{F}_q$ if and only if $q\equiv\pm2\pmod5$ and
 $$\left(\frac{2a^2}{u}\right)^2=5\cdot\frac{4a^5-2-\gamma}{4a u^2}.$$
 That is $5\gamma=4a^2-10$, or $5\mid q$ and
 $$\left(\frac{a^2}{u}\right)^2=\frac{4a^5-2-\gamma}{4a u^2}.$$
That is $\gamma=-2\equiv3\pmod5$.
$\hfill\square$

{\bf Proof of Theorem \ref{Appx.4}~}
Since $q$ is an odd prime power, there is a non-square element $u\in\mathbb{F}_{q}$. Let $\alpha$ be an element of $\mathbb{F}_{q^2}$ with $\alpha^2=u$. Then $\{1, \alpha\}$ is a basis of $\mathbb{F}_{q^2}$ over $\mathbb{F}_q$. Let $\delta=a+b\alpha$ with $a, b\in\mathbb{F}_q$ and $x=y-(z-b)\alpha/2$. A simple computation shows that $2a=\Tr_{q}^{q^2}(\delta)$ and $x^q-x+\delta=a+z\alpha$. Then we have
\begin{align*}
f(x)&=(x^q-x+\delta)^{2q+3}+(x^q-x+\delta)^{5q}+\gamma x \\
  &=2a(a^2+3z^2u)(a-z\alpha)^2+\gamma(y-(z-b)\alpha/2)\\
  &=\gamma y+ 6au^2z^4+8a^3uz^2+2a^5-\left(12a^2uz^3+(4a^4+\frac{\gamma}{2})z-\frac{\gamma b}{2}\right)\alpha.
  \end{align*}
Hence
\begin{equation*}
\begin{cases}
g_1(y, z)=\gamma y+6au^2z^4+8a^3uz^2,\\
 g_2(y, z)=12a^2uz^3+(4a^4+\frac{\gamma}{2})z.
\end{cases}
\end{equation*}
Observe that $g_2(y, z)$ is a polynomial of $z$ and the $y$-part of $g_1(y, z)$ is $\gamma y$. Therefore $(g_1(y,z),g_2(y,z))$ permutes $\mathbb{F}_q^2$ if and only if $g_2(y, z)=g_2(z)$ is a PP over $\mathbb{F}_q$. It is clear that for $a=0$, the polynomial $g_2(z)=\gamma/2$ is a PP of $\mathbb{F}_q$ (since $\gamma\neq0$).
Now assume that $a\neq0$. If $3\mid q$, we have $g_2(z)=(a^4+\frac{\gamma}{2})z$, which is a PP of $\mathbb{F}_{q}$ if and only if $a^4+\frac{\gamma}{2}\neq0$, i.e., $\gamma\neq a^4$. If $3 \nmid q$,
the normalized form of $g_2(z)$  is
$$z^3+\frac{8a^4+\gamma}{24a^2 u}z.$$
In this case, according to Table 7.1 in \cite{LN97},  $g_2(z)$ is a PP over $\mathbb{F}_q$ if and only if  $q\not\equiv1\pmod{3}$ and $8a^4+\gamma=0$.

In summary, $(g_1(y,z)), g_2(y,z)$ permutes $\mathbb{F}_q^2$ if and only if $a=0$, or $a\neq0 ,3\mid q$ and $\gamma\neq a^4$, or $a\neq0$, $q\equiv2\pmod3$ and  $\gamma=-8a^4$. Applying Lemma \ref{lem-n=2}, we complete the proof.
$\hfill\square$

{\bf Proof of Theorem \ref{Appx.5}~}
Since $q$ is an odd prime power, there is a non-square element $u\in\mathbb{F}_{q}$. Let $\alpha$ be an element of $\mathbb{F}_{q^2}$ with $\alpha^2=u$. Then $\{1, \alpha\}$ is a basis of $\mathbb{F}_{q^2}$ over $\mathbb{F}_q$. Let $\delta=a+b\alpha$  with $a, b\in\mathbb{F}_q$ and $x=y-(z-b)\alpha/2$. Then $2a=\Tr_{q}^{q^2}(\delta)$, $x^q-x+\delta=a+z\alpha$, and we have
\begin{align*}
f(x) &= (x^q-x+\delta)^{2q+4}+(x^q-x+\delta)^{2q}+\gamma x\\
 &=(a^2-z^2u)^2(a+z\alpha)^2+(a-z\alpha)^2+\gamma(y-(z-b)\alpha/2)\\
  &=\gamma y+u^3z^6-a^2u^2z^4+(1-a^4)uz^2+a^6+a^2+\left(2au^2z^5-4a^3uz^3+\left(2a^5-2a-\frac{\gamma }{2}\right) z+\frac{\gamma b}{2}\right)\alpha.
\end{align*}
Hence
\begin{equation*}
\begin{cases}
g_1(y, z)=\gamma y+u^3z^6-a^2u^2z^4+(1-a^4)uz^2,\\
g_2(y, z)=2au^2z^5-4a^3uz^3+\left(2a^5-2a-\frac{\gamma }{2}\right)z.
\end{cases}
\end{equation*}
Observe that $g_2(y, z)$ is a polynomial of $z$ and the $y$-part of $g_1(y, z)$ is $\gamma y$. Therefore $(g_1(y,z),g_2(y,z))$ permutes $\mathbb{F}_q^2$ if and only if $g_2(y, z)=g_2(z)$ is a PP over $\mathbb{F}_q$. Note that if $a=0$,  we have $g_2(z)=-\frac{\gamma}{2}z$, which clearly is a PP over $\mathbb{F}_{q}$ since $\gamma\neq0$. Now assume $a\neq0$. Then
the normalized form of $g_2(z)$  is
$$z^5-\frac{2a^2}{u}z^3+\frac{4a^5-4a-\gamma}{4au^2}z.$$
From Table 7.1 in \cite{LN97}, $g_2(z)$ is a PP over $\mathbb{F}_q$ if and only if $ q\equiv\pm2\pmod{5}$ and
$$\left(\frac{-2a^2}{u}\right)^2=5 \cdot \frac{4a^5-4a-\gamma}{4au^2},$$
i.e., $5\gamma=4a^5-20a$, or $5\mid q$ and
$$\left(\frac{a^2}{u}\right)^2=\frac{4a^5-4a-\gamma}{4au^2},$$ i.e., $\gamma=-4a$.
$\hfill\square$

{\bf Proof of Theorem \ref{Appx.6}~}
Since $q$ is an odd prime power, there is a non-square element $u\in\mathbb{F}_{q}$. Let $\alpha$ be an element of $\mathbb{F}_{q^2}$ with $\alpha^2=u$. Then $\{1, \alpha\}$ is a basis of $\mathbb{F}_{q^2}$ over $\mathbb{F}_q$. Let $\delta=a+b\alpha$ with $a, b\in\mathbb{F}_q$ and $x=y-(z-b)\alpha/2$. A simple computation shows that $2a=\Tr_{q}^{q^2}(\delta)$ and $x^q-x+\delta=a+z\alpha$. Then we have
\begin{align*}
f(x) &= (x^q-x+\delta)^{ q+5}+(x^q-x+\delta)^{2 q}+\gamma x\\
&=(a^2-z^2u)(a+z\alpha)^4+(a-z\alpha)^2+\gamma(y-(z-b)\alpha/2)\\
&=\gamma y-u^3z^6-5a^2u^2z^4+(5a^4+1)uz^2+a^6+a^2+\left(-4au^2z^5+\left(4a^5-2a-\frac{\gamma}{2}\right)z+\frac{\gamma b}{2}\right)\alpha.
\end{align*}
Hence
\begin{equation*}
\begin{cases}
g_1(y, z)=\gamma y-u^3z^6-5a^2u^2z^4+(5a^4+1)uz^2,\\
g_2(y,z)=-4au^2z^5+\left(4a^5-2a-\frac{\gamma}{2}\right)z.
\end{cases}
\end{equation*}
Observe that $g_2(y, z)$ is a polynomial of $z$ and the $y$-part of $g_1(y, z)$ is $\gamma y$. Therefore $(g_1(y,z),g_2(y,z))$ permutes $\mathbb{F}_q^2$ if and only if $g_2(y, z)=g_2(z)$ is a PP over $\mathbb{F}_q$. When $a=0$, it is clear that $g_2(z)=-\frac{\gamma}{2}z$ is a PP over $\mathbb{F}_q$ as $\gamma\ne0$. When $a\ne0$, the normalized form  of $g_2(z)$  is
$$z^5-\frac{8a^5-4a-\gamma}{8au^2}z.$$
In this case, by Table 7.1 in \cite{LN97}, the polynomial $g_2(z)$ is a PP over $\mathbb{F}_q$ if and only if $q\not\equiv1\pmod5$ and $8a^5-4a-\gamma=0$, i.e., $\gamma=8a^5-4a$, or $5\mid q$ and
$$\frac{8a^5-4a-\gamma}{8a}$$
is a fourth power or not a square, or $q=9$ and $$2=\left(\frac{\gamma-8a^5+4a}{8au^2}\right)^2\equiv\left(\frac{\gamma+a^5+a}{2au^2}\right)^2\pmod3,$$
 which implies $(\gamma+a^5+a)^2=2a^2u^4$.
 Since $u^4=2$ in $\mathbb{F}_9$ ((as shown in the proof of Theorem \ref{Licao2023th}), we obtain $(\gamma+a^5+a)^2=a^2$. Taking square roots yields  $\gamma+a^5+a=\pm a$, and therefore
 $\gamma=2a^5$ or $\gamma=2a^5-2a$.

$\hfill\square$

{\bf Proof of Theorem \ref{Liu-Jiang-Zou2}~}
 Since $q$ is an odd prime power,  there is a non-square element $u\in\mathbb{F}_{q}$. Let $\alpha$ be an element of $\mathbb{F}_{q^2}$ with $\alpha^2=u$. Then $\{1, \alpha\}$ is a basis of $\mathbb{F}_{q^2}$ over $\mathbb{F}_q$. Let $\delta=a+b\alpha$ with $a, b\in\mathbb{F}_q$ and $x=y-(z-b)\alpha/2$. Then $2a=\Tr_q^{q^2}(\delta)$ and $x^q-x+\delta=a+z\alpha$. We have
\begin{align*}
f(x)&=(x^q-x+\delta)^{q+2}+\gamma(x^q+x)\\
  &=(a^2-z^2u)(a+z\alpha)+2\gamma y\\
   &=2\gamma y-auz^2+a^3+(-u z^3 +a^2z)\alpha.
\end{align*}
Hence
\begin{equation*}
\begin{cases}
g_1(y, z)=2\gamma y-auz^2,\\
g_2(y, z)=-u z^3 +a^2z.
\end{cases}
\end{equation*}
Observe that $g_2(y, z)$ is a polynomial of $z$ and the $y$-part of $g_1(y, z)$ is $2\gamma y$. Hence, $(g_1(y,z),g_2(y,z))$ permutes $\mathbb{F}_q^2$ if and only if $g_2(y, z)=g_2(z)$ is a PP over $\mathbb{F}_q$.
Since $u$ is a non-square, it follows from Table 7.1 in \cite{LN97} that $g_2(z)$ is a PP over $\mathbb{F}_q$ if and only if $q\not\equiv1\pmod{3}$ and $a=0$, or $3\mid q$ and $a\ne0$.  Therefore, $f(x)$ is a PP of $\mathbb{F}_{q^2}$ if and only if $q\not\equiv1\pmod{3}$ and $\Tr_q^{q^2}(\delta)=0$, or $q\equiv0\pmod3$ and $\Tr_q^{q^2}(\delta)\ne0$.
$\hfill\square$

{\bf Proof of Theorem \ref{Liu-Jiang-Zou3}~}
Since $q$ is an odd prime power, there is a non-square element $u\in\mathbb{F}_{q}$. Let $\alpha$ be an element of $\mathbb{F}_{q^2}$ with $\alpha^2=u$. Then $\{1, \alpha\}$ is a basis of $\mathbb{F}_{q^2}$ over $\mathbb{F}_q$. Let $\delta=a+b\alpha$ with $a, b\in\mathbb{F}_q$ and $x=y-(z-b)\alpha/2$. Then $2a=\Tr_q^{q^2}(\delta)$ and $x^q-x+\delta=a+z\alpha$. We have
\begin{align*}
f(x)&=(x^q-x+\delta)^{3q+2}+\gamma(x^q+x)\\
&=(a^2-z^2u)^2(a-z\alpha)+2\gamma y\\
&=2\gamma y+au^2z^4-2a^3uz^2+a^5+\left(-u^2z^5+2a^2uz^3-a^4z\right)\alpha.
\end{align*}
Hence
\begin{equation*}
\begin{cases}
g_1(y, z)=2\gamma y+au^2z^4-2a^3uz^2,\\
g_2(y, z)=-u^2z^5+2a^2uz^3-a^4z.
\end{cases}
\end{equation*}
Observe that $g_2(y, z)$ is a polynomial of $z$ and the $y$-part of $g_1(y, z)$ is $2\gamma y$. Hence, $(g_1(y,z),g_2(y,z))$ permutes $\mathbb{F}_q^2$ if and only if $g_2(y, z)=g_2(z)$ is a PP over $\mathbb{F}_q$. Note that the  normalized form of $g_2(z)$ is as follows
$$z^5-\frac{2a^2}{u}z^3+\frac{a^4}{u^2}z.$$
By Table 7.1 in \cite{LN97}, $g_2(z)$ is a PP over $\mathbb{F}_q$ if and only if $q\not\equiv1\pmod{5}$ and $a=0$, or $5\mid q$ and $a\ne0$. It  follows that $f(x)$ is a PP of $\mathbb{F}_{q^2}$ if and only if $q\not\equiv1\pmod{5}$ and $\Tr_q^{q^2}(\delta)=0$, or $q\equiv0\pmod5$ and $\Tr_q^{q^2}(\delta)\ne0$.
$\hfill\square$

{\bf Proof of Theorem \ref{Liu-Jiang-Zou4}~}
 Since $q$ is an odd prime power,  there is a non-square element $u\in\mathbb{F}_{q}$. Let $\alpha$ be an element of $\mathbb{F}_{q^2}$ with $\alpha^2=u$. Then $\{1, \alpha\}$ is a basis of $\mathbb{F}_{q^2}$ over $\mathbb{F}_q$. Let $\delta=a+b\alpha$ with $a, b\in\mathbb{F}_q$ and $x=y-(z-b)\alpha/2$. Then $2a=\Tr_q^{q^2}(\delta)$ and  $x^q-x+\delta=a+z\alpha$. We have
\begin{align*}
f(x)&=(x^q-x+\delta)^{4q+2}+\gamma(x^q+x)\\
&=(a^2-z^2u)^2(a-z\alpha)^2+2\gamma y\\
  &=2\gamma y+u^3z^6-a^2u^2z^4-a^4uz^2+a^6+\left(-2au
  ^2z^5+4a^3uz^3-2a^5z\right)\alpha.
 \end{align*}
Hence
\begin{equation*}
\begin{cases}
g_1(y, z)=2\gamma y-u^3z^6-a^2u^2z^4-a^4uz^2,\\
g_2(y, z)=-2au
  ^2z^5+4a^3uz^3-2a^5z.
\end{cases}
\end{equation*}
%
Observe that $g_2(y, z)$ is a polynomial of $z$ and the $y$-part of $g_1(y, z)$ is $2\gamma y$. Hence, $(g_1(y,z),g_2(y,z))$ permutes $\mathbb{F}_q^2$ if and only if $g_2(y, z)=g_2(z)$ is a PP over $\mathbb{F}_q$. Clearly, if $a=0$, then $g_2(z)=0$ for all $z$, so $g_2(z)$ is not a  PP over $\mathbb{F}_q$.  We therefore assume $a\neq 0$ in the following. Then $g_2(z)$ can be normalized as
$$z^5-\frac{2a^2}{u}z^3+\frac{a^4}{u^2}z.$$
The following discussion proceeds exactly as in Theorem \ref{Liu-Jiang-Zou3}. Therefore, $f(x)$ is a PP over $\mathbb{F}_{q^2}$ if and only if $q\equiv0\pmod5$ and $\Tr_q^{q^2}(\delta)\neq0$.
$\hfill\square$

{\bf Proof of Theorem \ref{Liu-Jiang-Zou6}~}
Since $q$ is an odd prime power, there is a non-square element $u\in\mathbb{F}_{q}$. Let $\alpha$ be an element of $\mathbb{F}_{q^2}$ with $\alpha^2=u$. Then $\{1, \alpha\}$ is a basis of $\mathbb{F}_{q^2}$ over $\mathbb{F}_q$. Let $\delta=a+b\alpha$ with $a, b\in\mathbb{F}_q$ and $x=y-(z-b)\alpha/2$. Then $2a=\Tr_q^{q^2}(\delta)$ and  $x^q-x+\delta=a+z\alpha$. We have
\begin{align*}
f(x)&=(x^q-x+\delta)^{3q+2}+(x^q-x+\delta)^{4q+2}+\gamma(x^q+x)\\
&=(a^2-z^2u)^2(a-z\alpha)+(a^2-z^2u)^2(a-z\alpha)^2+2\gamma y\\
  &=2\gamma y+u^3z^6+au^2(1-a)z^4-a^3u(a+2)z^2+a^6+a^5+(2a+1)(u^2z^5-2a^2uz^3+a^4z)\alpha.
\end{align*}
Hence
\begin{equation*}
\begin{cases}
g_1(y, z)=2\gamma y+u^3z^6+au^2(1-a)z^4-a^3u(a+2)z^2,\\
 g_2(y, z)=(2a+1)(u^2z^5-2a^2uz^3+a^4z).
\end{cases}
\end{equation*}
Observe that $g_2(y, z)$ is a polynomial of $z$ and the $y$-part of $g_1(y, z)$ is $2\gamma y$. Hence $(g_1(y,z),g_2(y,z))$ permutes $\mathbb{F}_q^2$ if and only if $g_2(y, z)=g_2(z)$ is a PP over $\mathbb{F}_q$. If $2a+1=0$, the $g_2(z)=0$ cannot be a PP over $\mathbb{F}_q$. So assume that $2a+1\ne0$. Then the normalized form of $g_2(z)$ is
$$z^5-\frac{2a^2}{u}z^3+\frac{a^4}{u^2}z.$$
From Table 7.1 in \cite{LN97}, the polynomial
$g_2(z)$ is a PP over $\mathbb{F}_q$ if and only if $q\not\equiv1\pmod{5}$ and $a=0$, or $5\mid q$ and $a\neq 0$. Combining with $2a+1\ne0$, it follows that
$f(x)$ is a PP of $\mathbb{F}_{q^2}$ if and only if  $q\not\equiv1\pmod{5}$ and $\Tr_q^{q^2}(\delta)=0$, or $q\equiv0\pmod5$ and  $\Tr_q^{q^2}(\delta)\not\in\{-1,0\}$.
$\hfill\square$

\end{document}